\documentclass[11pt]{article}

\usepackage[utf8]{inputenc}
\usepackage{amsmath,amssymb,amsthm,mathtools,mathrsfs}
\usepackage[margin=1in]{geometry}
\usepackage{microtype}
\usepackage{graphicx}
\usepackage{array,booktabs,calc,longtable}
\usepackage{enumitem}
\usepackage{fvextra}
\usepackage[numbers,sort&compress]{natbib}
\usepackage{xurl}
\usepackage[hidelinks,pdfusetitle]{hyperref}

\allowdisplaybreaks
\newtheorem{theorem}{Theorem}[section]
\newtheorem{lemma}[theorem]{Lemma}
\newtheorem{proposition}[theorem]{Proposition}
\newtheorem{corollary}[theorem]{Corollary}
\theoremstyle{definition}

\theoremstyle{remark}

\newtheorem*{mainresult}{Main result}

\DefineVerbatimEnvironment{CertificateOutput}{Verbatim}{%
  breaklines=true,
  breakanywhere=true,
  fontsize=\footnotesize,
  frame=none,
  baselinestretch=0.96
}

\newcommand{\manuscriptauthorgiven}{Ian Meng}
\newcommand{\manuscriptauthorfamily}{Si}
\newcommand{\manuscriptauthoraffiliation}{Department of Statistics, University of California, Berkeley}
\newcommand{\manuscriptauthoremail}{ianmeng0511@berkeley.edu}

\title{Strategic geometry of competing first-passage random walks}
\author{%
  \manuscriptauthorgiven~\manuscriptauthorfamily\\
  \manuscriptauthoraffiliation\\
  \texttt{\manuscriptauthoremail}%
}
\date{}
\hypersetup{
  pdfauthor={Ian Meng Si},
  pdfsubject={Strategic first-passage competition of independent random walks},
  pdfkeywords={random walks, first passage, reflected Brownian motion, maximal lotteries, equilibrium support}
}

\begin{document}
\maketitle

\begin{abstract}
Two competitors choose starting vertices for independent, constant-speed
random walks, and each site is acquired by its first visitor. We study the
spatial geometry of the resulting first-passage location game. On every finite
path the optimal strategies are exactly the distributions supported on the
central vertices. The proof combines reflecting-boundary harmonic barriers
with a parameter-uniform aggregate estimate for product-chain exit
probabilities. After diffusive rescaling, the complete two-start payoff
landscape converges uniformly to the game between independent reflected
Brownian motions. Its unique equilibrium concentrates at the midpoint, with
explicit cubic stability. Beyond paths, attaching two leaves to every vertex of a clique of
order \(k\) produces a \(3k\)-vertex graph on which every exact optimal
strategy randomizes over all \(k\) clique vertices; for \(k=2\), this is a
six-vertex tree with no pure equilibrium. The continuum best response to an
endpoint is uniquely determined. A first-passage random-ranking
representation relates the finite game to maximal lotteries without
identifying it with nonstrategic painting, deterministic Voronoi allocation,
or absorbing-trap placement. Parameter-uniform statements follow from
analytic arguments or symbolic polynomial identities; identified finite
exceptions and numerical enclosures have reproducible certificates.
\end{abstract}


\tableofcontents

\section*{Introduction}
\addcontentsline{toc}{section}{Introduction}

Competitive exploration becomes strategic when the competitors can choose
where their random motion begins. A resource at a site is secured by whichever
mobile agent discovers that site first, and each agent chooses its initial
location before the walks evolve. The resulting question is not simply how
much territory fixed walkers acquire: it is how the geometry of independent
first-passage processes determines strategically stable spatial positions and,
when pure positions fail, the necessary structure of randomization.

This formulation supplies a tractable probabilistic model for territorial
acquisition, competitive exploration, mobile-agent resource allocation, and
location decisions under stochastic movement. These are motivations for the
model, not claims of empirical validation. We work throughout with independent
\emph{constant-speed} continuous-time walks on finite graphs: every
nonisolated vertex, including a reflecting path endpoint or pendant leaf, has
total jump rate one.

For each target vertex, ownership is a race between two independent hitting
times and is governed by a product-chain Dirichlet problem. Aggregating those
first-passage probabilities over the graph yields a constant-sum location game
whose payoff matrix is constrained by the underlying walk. The probabilistic
constraints, rather than abstract minimax theory, determine the results. On
paths they force exact central optimality; under diffusive rescaling they
produce uniform convergence of the \emph{entire} strategic payoff landscape to
a reflected-Brownian game; and on other graphs they can destroy pure
equilibrium and require support linear in graph size.

The proof mechanisms mirror this progression. Reflecting-boundary harmonic
comparisons, strict coupled rectangle domination, and a parameter-uniform
aggregate exit estimate control the finite path. A varying-start and
varying-target first-passage limit promotes an invariance principle to uniform
strategic convergence and quantitative continuum stability. Explicit lumped
product-chain calculations then show that first-passage-induced preference
matrices can realize large \emph{exact} equilibrium support. The endpoint
optimization illustrates that the continuum payoff retains genuinely
stochastic spatial geometry; its numerical localization is a consequence of
the first-passage analysis, not the principal contribution.

\subsection*{Related work and model distinctions}

First-arrival territorial coloring is established prior work. Gomes J\'{u}nior,
Lucena, da Silva, and Hilhorst~\cite{gomes1996coloring} already studied two
independent walkers coloring a one-dimensional periodic lattice and analyzed
its scaling behavior. Miller~\cite{miller2013painting} developed the
competing-walk painting problem on graphs with nonstrategic random
initialization. Most closely in claiming rule, Chatterjee, Nabahi, and
Terlov~\cite{chatterjee2026interface} study two fixed-start independent walks
on a cycle and the expected number of interfaces between their claimed
territories. Plaud, Krapivsky, Redner, and
B\'{e}nichou~\cite{plaud2026discoveries} investigate fluctuations of discovery
shares in competitive exploration, while Baccara~\cite{baccara2026range}
studies fixed-start competitive ranges on infinite lattices. Dey, Kim, and
Terlov~\cite{dey2025collaboration} instead study collaborative exploration.
None of these distinctions makes first-arrival coloring or one-dimensional
scaling new; the additional feature here is simultaneous strategic choice of
starting positions and the analysis of the induced equilibrium problem.

Deterministic spatial games supply another important comparison. D\"urr and
Nguyen~\cite{durr2007voronoi}, and subsequently
Bandyapadhyay et al.~\cite{bandyapadhyay2014voronoi}, analyze graph Voronoi
games in which customers are allocated by nearest-facility distances.
Roshanbin~\cite{roshanbin2014diffusion} proves pure-equilibrium existence on
trees for a deterministic competitive-diffusion game. Independent
first-passage allocation is neither deterministic nearest-facility assignment
nor deterministic diffusion: indeed, the six-vertex double-star tree below
has no pure equilibrium in the present model. This contrast concerns
different allocation dynamics and is not a contradiction of Roshanbin's
theorem.

Strategic positioning for random-walk competition also appears in a different
probabilistic model.  Cencetti, Bagnoli, Di Patti, and
Fanelli~\cite{cencetti2016second} and Baldi and
Bagnoli~\cite{baldi2020intransitiveness} study competing absorbing traps.  With
initial distribution \(\nu\), their representative two-position kernel has the
form
\[
W_{\mathrm{trap}}(a,b)
=\sum_z\nu(z)\Pr_z(T_a<T_b),
\]
so that \emph{one walker} races toward two strategically placed sites.  The
present territorial kernel is
\[
W_{\mathrm{territory}}(a,b)
=\frac1{|V|}\sum_z
\Pr\!\left(T_z^{(a)}<T_z^{(b)}\right),
\]
so that \emph{two independent walkers} race toward each individual target.
Half ties are added where appropriate.  The distinction already holds on
\(P_3\), with \(a=1\), \(b=0\), and uniform initial/target mass:
\[
W_{\mathrm{trap}}(1,0)=\frac23,
\qquad
W_{\mathrm{territory}}(1,0)=\frac{13}{24}.
\]
Thus strategic trap placement is established prior work, but its payoff kernel
is not the kernel analyzed here.

At the abstract expected-payoff level, the finite game lies inside existing
location-game frameworks. N\'{u}\~nez and
Scarsini~\cite{nunez2016locations} treat competition over finitely many
locations under strict consumer rankings, and Frongillo, Hsu, Monroe, and
Thilagar~\cite{frongillo2026position} develop a broad position-optimization
framework.  The maximal-lottery interpretation is likewise established
theory, beginning with Fishburn~\cite{fishburn1984probabilistic} and developed
by Brandl, Brandt, and Seedig~\cite{brandl2016consistent}.  Our representation
is an explicit \emph{first-passage-induced embedding} into these frameworks,
not a claim to have introduced them.

Large exact supports are also known for unconstrained majority games; see,
for example, Brandl, Brandt, and
Stricker~\cite{brandl2022maximal}.  What is specific here is a natural
\(3k\)-vertex graph whose first-passage constraints force every exact
optimal strategy to use precisely \(k\) vertices.  Even support sizes cause no
conflict with odd-electorate or genericity statements in the social-choice
literature: the induced hub--hub majority margins are tied and the optimal
lottery is not unique.  Moreover, the result is exclusively about \emph{exact}
optimality.  Small-support approximate maximal lotteries exist in general
\cite{charikar2025dominating}, and a single hub is already an
\(O(k^{-2})\)-optimal normalized strategy for our example.

The resulting novelty claim is deliberately narrow:
\emph{strategic starting-position selection for independent first-arrival
territorial walks yields a graph-constrained location game with a complete
finite-path equilibrium classification, a uniform reflected-Brownian
strategic limit, and explicit graph families requiring linear exact
equilibrium support.} Neither competing-walk painting itself, general
one-dimensional scaling, deterministic graph Voronoi games, maximal lotteries,
abstract large-support majority games, nor strategic trap placement is claimed
to be new.

\subsection*{Main results}

For a finite graph \(G\), let \(U_G(a,b)\) be the expected number of vertices
won by the first player, and write
\(A_G(a,b)=U_G(a,b)-|V(G)|/2\).  Denote the set of optimal mixed strategies by
\(\mathcal O(G)\).

\begin{mainresult}[Complete path equilibria]
For every path, \(\mathcal O(P_n)\) is the set of probability measures
supported on its center.  In particular,
\[
\mathcal O(P_{2m+1})=\{\delta_m\},
\qquad
\mathcal O(P_{2m})
=\left\{p\delta_{m-1}+(1-p)\delta_m:0\leq p\leq1\right\}.
\]
The pure equilibria are respectively \(\{(m,m)\}\) and
\(\{m-1,m\}^2\), and all mixed equilibria are
\(\mathcal O(P_n)\times\mathcal O(P_n)\).  The assertions include \(P_1\) and
\(P_2\).
\end{mainresult}

\begin{mainresult}[Uniform Brownian scaling and equilibrium stability]
Let \(F(x,y)\) be the target-integrated first-arrival payoff for independent
standard reflected Brownian motions on \([0,1]\).  For paths indexed by
\(\{0,\dots,N\}\),
\[
\sup_{0\leq a,b\leq N}
\left|
\frac{U_{P_{N+1}}(a,b)}{N+1}
-F\!\left(\frac aN,\frac bN\right)
\right|\longrightarrow0.
\]
The unique optimal continuum strategy is \(\delta_{1/2}\), and
\[
F\!\left(\frac12,y\right)-\frac12
\geq\frac{|1-2y|^3}{400}.
\]
In particular, the unique continuum mixed equilibrium is
\((\delta_{1/2},\delta_{1/2})\), and the cubic bound yields explicit
concentration estimates for approximate equilibria.
\end{mainresult}

\begin{mainresult}[Graph-induced exact equilibrium support and a tree obstruction]
Let \(G_{k,2}\) be the graph formed from a \(k\)-clique by attaching two
pendant leaves to each clique vertex \(u_1,\ldots,u_k\).  For every \(k\geq2\)
and every exact optimal strategy \(\mu\),
\[
\operatorname{supp}(\mu)=\{u_1,\ldots,u_k\},
\qquad
|\operatorname{supp}(\mu)|
=k=\frac{|V(G_{k,2})|}{3}.
\]
More generally, all optimal strategies and equilibria of the clique-of-hubs
graphs \(G_{k,r}\), for all \(k\geq2\) and \(r\geq1\), are classified by
explicit rational first-passage formulas. In particular, \(G_{2,2}\) is a
six-vertex double-star tree with no pure equilibrium: every optimal strategy
assigns strictly positive probability to both hubs. Its two centered
first-passage payoff parameters are
\[
h_{2,2}=\frac{385}{1812},
\qquad
d_{2,2}=\frac{19}{755}.
\]
\end{mainresult}

\begin{mainresult}[Certified continuum endpoint response]
The function \(H(x)=F(x,0)\) has a unique global maximizer \(x_*\), with
\[
0.363424228<x_*<0.363424229,
\qquad
0.57155664213<H(x_*)<0.57155664221.
\]
Every sequence of finite-path endpoint best responses \(a_N\) satisfies
\(a_N/N\to x_*\). No uniqueness assertion for the finite optimizers and no
convergence rate are made.
\end{mainresult}

\medskip\noindent\textbf{Supporting random-ranking representation.}
For every finite connected graph, a random strict ranking \(\Pi\) of its
vertices has uniform first-ranked vertex and majority margins
\[
\Pr_{\Pi}(a\succ b)-\Pr_{\Pi}(b\succ a)
=\frac{2A_G(a,b)}{|V(G)|}.
\]
Consequently, \(\mathcal O(G)\) is the maximal-lottery set of this ranking
distribution. This structural proposition concerns two-player expected
payoffs; it neither identifies the joint painting processes nor extends
unchanged to the stated three-walker example.

\subsection*{Proof architecture and computational boundaries}

The common first-passage mechanism organizes the arguments into four dependent
parts.

\noindent\textbf{Finite paths.}
Sections~4--6 pass from the product-chain Dirichlet problem to harmonic
barriers respecting the rate-one reflecting boundary. An outer-flux
comparison and a strict folding-coupling domination reduce the interior
deficit to ordinary rectangle exits. A parameter-uniform aggregate exit
estimate, followed by the explicitly enumerated endpoint exceptions, proves
strict centrality and then both the pure and mixed equilibrium
classifications.

\noindent\textbf{Reflected-Brownian limit.}
Section~7 combines the folded invariance principle with a varying-start and
varying-target hitting-time lemma. Boundary targets, coincident limiting
starts, and target neighborhoods are handled separately. Continuity of the
integrated payoff and compactness then give uniform two-start convergence;
the finite harmonic bounds yield cubic midpoint stability and uniqueness of
the continuum mixed equilibrium.

\noindent\textbf{Graph support.}
Section~3 states the structural consequence, and Appendix~A proves the full
clique-of-hubs classification. Graph automorphisms reduce the product-chain
systems to two weighted orbit matrices. Exact symbolic identities and
positive shifted coefficients hold for every parameter; specialization to two
leaves gives linear exact support and its six-vertex tree corollary.

\noindent\textbf{Endpoint response.}
Section~8 first globally localizes every maximizer. An image expansion
regularizes the moving target diagonal; uniform exponential domination
justifies integration, differentiation, and cancellation of boundary terms.
Explicit image-tail and quadrature remainders give strict concavity and
certified opposite derivative signs, yielding uniqueness and the stated
numerical enclosures.

Appendix~B distinguishes five kinds of evidence: analytic arguments valid for
infinitely many parameters; exact symbolic polynomial identities in those
parameters; the theorem-critical finite exceptional cases; certified
outward-rounded numerical enclosures; and finite regression sweeps. Only the
first two establish parameter-uniform claims. The four executable supplements
document the explicitly identified finite exceptions and numerical
certifications without substituting finite sweeps for proofs.

\hypertarget{model-payoff-normalization-and-constant-sum-game}{%
\section{Model, payoff normalization, and constant-sum game}\label{model-payoff-normalization-and-constant-sum-game}}

Let \(G=(V,E)\) be a finite connected graph and let \(n=|V|\). At every
nonisolated vertex \(v\), the walk waits an exponential time of mean one and
then chooses a uniformly random neighbor. Its generator is

\[
(\mathcal L_G f)(v)
=\frac{1}{\deg(v)}
 \sum_{w\sim v}\bigl(f(w)-f(v)\bigr).
\tag{1.1}
\]

In particular, an endpoint of a path still has total jump rate one: its unique
neighbor is selected with probability one. For the single-vertex graph the
walk stays at that vertex. Ownership does not prevent subsequent traversal.

Given independent walks \(X^{(a)}\) and \(X^{(b)}\), define

\[
\tau_z^{(a)}=\inf\{t\geq 0:X_t^{(a)}=z\}
\]

and

\[
q_z(a,b)
=\Pr\bigl(\tau_z^{(a)}<\tau_z^{(b)}\bigr)
+\frac12\Pr\bigl(\tau_z^{(a)}=\tau_z^{(b)}\bigr).
\tag{1.2}
\]

The unnormalized expected payoff and centered payoff are

\[
U_G(a,b)=\sum_{z\in V}q_z(a,b),
\qquad
A_G(a,b)=U_G(a,b)-\frac{|V(G)|}{2}.
\tag{1.3}
\]

Co-location is permitted and gives \(q_z(a,a)=1/2\) in expectation for every
target, so \(U_G(a,a)=n/2\). Positive first-passage times have atomless laws:
conditional on the jump count of first arrival, they are sums of independent
exponential holding times. Consequently distinct starting vertices have no
positive-time ties.

For every \(a,b\in V\),

\[
U_G(a,b)+U_G(b,a)=n,
\qquad
A_G(a,b)=-A_G(b,a).
\tag{1.4}
\]

The zero-sum value of \(A_G\) is zero. Write

\[
C_G=\{a\in V:A_G(a,b)\geq 0\text{ for every }b\in V\}
\]

and

\[
\mathcal O(G)
=\left\{
\mu\in\mathcal P(V):
\sum_{a\in V}\mu(a)A_G(a,b)\geq 0
\text{ for every }b\in V
\right\}.
\tag{1.5}
\]

The minimax theorem and skew symmetry give

\[
\operatorname{PNE}(G)=C_G\times C_G,
\qquad
\operatorname{MNE}(G)=\mathcal O(G)\times\mathcal O(G).
\tag{1.6}
\]

The path \(P_{N+1}\) always has vertices \(0,1,\ldots,N\), hence \(N\) edges
and \(N+1\) vertices. Its spatial scaling is \(a/N\) when \(N\geq 1\), whereas
its normalized payoff is \(U_{P_{N+1}}/(N+1)\). These two normalizations are
not interchangeable at finite \(N\).

The first-arrival probabilities are solutions of the product-chain Dirichlet
problem. Away from either absorbing target coordinate,

\[
(\mathcal L_G^{(1)}+\mathcal L_G^{(2)})q_z=0,
\tag{1.7}
\]

with boundary values one when only the first coordinate is the target, zero
when only the second coordinate is the target, and \(1/2\) when both
coordinates coincide. Equation (1.7) retains the degree-dependent transition
probabilities of (1.1); in particular it does not replace endpoint rate one by
an artificial holding or endpoint slowdown.

\hypertarget{first-passage-random-rankings-and-their-precise-limitation}{%
\section{First-passage random rankings and their precise limitation}\label{first-passage-random-rankings-and-their-precise-limitation}}

\begin{proposition}[First-passage random-ranking representation]
\label{thm:random-ranking}
For every finite connected graph \(G=(V,E)\), there is a random strict ranking
\(\Pi\) of \(V\) whose first-ranked vertex is uniform and whose majority-margin
matrix satisfies
\[
\Pr_{\Pi}(a\succ b)-\Pr_{\Pi}(b\succ a)
=\frac{2A_G(a,b)}{|V|}.
\]
The optimal graph-game strategies are precisely the maximal lotteries of this
ranking distribution.
\end{proposition}

\begin{proof}
Choose \(Z\) uniformly from \(V\). Conditional on \(Z=z\), independently sample
for every alternative \(x\in V\) a variable

\[
T_x^{(z)}\stackrel{\mathrm d}{=}\tau_z^{(x)},
\qquad
T_z^{(z)}=0.
\tag{2.1}
\]

The variables with \(x\neq z\) are positive and atomless, so they almost surely
induce a strict ranking \(\Pi\) by increasing sampled time. The first-ranked
alternative is exactly \(z\). Therefore

\[
\Pr\bigl(\Pi(1)=z\bigr)=\frac1{|V|}.
\tag{2.2}
\]

For distinct \(a,b\), conditional independence and (1.2) give

\[
\Pr_\Pi(a\succ b)
=\frac1{|V|}\sum_{z\in V}
  \Pr\bigl(T_a^{(z)}<T_b^{(z)}\bigr)
=\frac{U_G(a,b)}{|V|}.
\tag{2.3}
\]

Thus the majority-margin matrix of the random strict ranking is

\[
M_G(a,b)
=\Pr_\Pi(a\succ b)-\Pr_\Pi(b\succ a)
=\frac{2A_G(a,b)}{|V|}.
\tag{2.4}
\]

Co-location agrees at the level of expected payoffs because both models divide
a colocated target share equally. Consequently \(\mathcal O(G)\) is exactly the
set of maximal lotteries of the induced random-ranking profile.
\end{proof}

This is an equivalence of two-player expected-payoff games. It is not an
equivalence of the complete joint random territorial painting process, and the
same shared-ranking construction does not reproduce natural three-walker
competition when starts are repeated. On \(P_3\), for starting profile
\((0,0,2)\), the conditional target shares are

\begin{table}[htbp]
\centering
\caption{The repeated-location obstruction on \(P_3\).}
\label{tab:three-walker-obstruction}
\begin{tabular}{@{}cll@{}}
\toprule
Target & Independent three-walker shares & Shared-ranking location shares\\
\midrule
\(0\) & \((1/2,1/2,0)\) & \((1/2,1/2,0)\)\\
\(1\) & \((1/3,1/3,1/3)\) & \((1/4,1/4,1/2)\)\\
\(2\) & \((0,0,1)\) & \((0,0,1)\)\\
\bottomrule
\end{tabular}
\end{table}

Indeed, at target \(1\) the three independent hitting times are identically
distributed, whereas the shared ranking first chooses which of the two distinct
locations wins and then divides a winning colocated location between its two
occupants. The resulting normalized individual payoffs are, respectively,

\[
\left(\frac5{18},\frac5{18},\frac49\right)
\quad\text{and}\quad
\left(\frac14,\frac14,\frac12\right).
\tag{2.5}
\]

This example rules out extension of this particular shared strict-ranking
representation; it does not rule out every generalized multiplayer model with
additional ties, player-specific proximity, or other latent randomness.

\hypertarget{graph-induced-linear-exact-equilibrium-support}{%
\section{Graph-induced linear exact equilibrium support}\label{graph-induced-linear-exact-equilibrium-support}}

For integers \(k\geq 2\) and \(r\geq 1\), let

\[
G_{k,r}
=K_k\text{ with }r\text{ pendant leaves attached to each clique vertex}.
\tag{3.1}
\]

Write \(u_1,\ldots,u_k\) for the hubs and \(L_i\) for the \(r\) leaves attached
to \(u_i\). Thus

\[
|V(G_{k,r})|=k(r+1),
\qquad
\deg(u_i)=\Delta=k+r-1,
\qquad
\deg(\ell)=1.
\tag{3.2}
\]

The complete all-\((k,r)\) classification and its product-chain proof appear in
Appendix~\ref{app:multihub}. Here its central consequence is stated separately.

\begin{theorem}[Linear exact equilibrium support]\label{thm:linear-support}
For every \(k\geq 2\), the graph \(G_{k,2}\) has \(3k\)
vertices and every exact optimal strategy \(\mu\) satisfies

\[
\operatorname{supp}(\mu)=\{u_1,\ldots,u_k\}.
\tag{3.3}
\]

Consequently

\[
\min_{\mu\in\mathcal O(G_{k,2})}|\operatorname{supp}(\mu)|
=\max_{\mu\in\mathcal O(G_{k,2})}|\operatorname{supp}(\mu)|
=k
=\frac{|V(G_{k,2})|}{3}.
\tag{3.4}
\]
\end{theorem}

\begin{proof}
More explicitly, let

\[
\begin{aligned}
W_k&=8k^3+46k^2+79k+47,\\
H_k&=8k^4+54k^3+121k^2+69k-27,\\
D_k&=8k^4+38k^3+45k^2+23k+26.
\end{aligned}
\tag{3.5}
\]

The two nonzero centered-payoff types are

\[
h_{k,2}=\frac{H_k}{2(k+4)W_k},
\qquad
d_{k,2}=\frac{D_k}{2(k+3)(k+4)W_k},
\tag{3.6}
\]

where \(h_{k,2}\) is a hub's advantage against one of its own leaves and
\(-d_{k,2}\) is its centered payoff against a leaf attached to another hub.
Appendix~\ref{app:multihub} proves that

\[
h_{k,2}-(k-1)d_{k,2}
=\frac{48k^4+276k^3+454k^2+177k-55}
       {2(k+3)(k+4)W_k}.
\tag{3.7}
\]

With \(p=k-2\), its numerator becomes

\[
5091+6841p+3262p^2+660p^3+48p^4>0.
\tag{3.8}
\]

Therefore \(h_{k,2}>(k-1)d_{k,2}>0\), and the full optimal-strategy set is

\[
\mathcal O(G_{k,2})
=\left\{
\sum_{i=1}^k x_i\delta_{u_i}:
\sum_{i=1}^k x_i=1,\
x_i\geq
\frac{D_k}{(k+3)H_k+D_k}
\text{ for every }i
\right\}.
\tag{3.9}
\]

Every coordinate in (3.9) is strictly positive, proving (3.3).
\end{proof}

\begin{corollary}[A six-vertex tree without a pure equilibrium]
\label{cor:tree-obstruction}
The double-star tree \(G_{2,2}\) has six vertices and no pure Nash
equilibrium. Its complete optimal-strategy set is
\[
\mathcal O(G_{2,2})
=\left\{p\delta_{u_1}+(1-p)\delta_{u_2}:
\frac{228}{2153}\leq p\leq\frac{1925}{2153}\right\}.
\]
In particular, every optimal strategy randomizes over both hubs.
\end{corollary}

\begin{proof}
At \(k=2\), (3.6) gives
\[
h_{2,2}=\frac{385}{1812},
\qquad
d_{2,2}=\frac{19}{755},
\qquad
\frac{d_{2,2}}{h_{2,2}+d_{2,2}}=\frac{228}{2153}.
\]
Substitution into (3.9) gives the displayed interval. Since neither hub may
receive zero mass and leaves receive none, no optimal strategy is pure;
(1.6) consequently rules out pure equilibrium.
\end{proof}

Roshanbin~\cite{roshanbin2014diffusion} proves existence of pure equilibria on
trees for a different game with deterministic diffusion. The six-vertex
example in Corollary~\ref{cor:tree-obstruction} shows that this conclusion need
not persist for independent first-arrival walks.

The support lower bound is solely an exact-optimality phenomenon. A pure hub
has worst unnormalized payoff deficit \(d_{k,2}\), and

\[
d_{k,2}\sim\frac1{2k},
\qquad
\frac{d_{k,2}}{|V(G_{k,2})|}
=\frac{d_{k,2}}{3k}
\sim\frac1{6k^2}.
\tag{3.10}
\]

Thus pure hubs themselves become highly accurate approximate optimal strategies
under normalized payoff. No large-support lower bound for approximate
equilibria is asserted.

\hypertarget{finite-paths-harmonic-kernels-and-reflecting-boundary-barriers}{%
\section{Finite paths: harmonic kernels and reflecting-boundary barriers}\label{finite-paths-harmonic-kernels-and-reflecting-boundary-barriers}}

For the rate-one reflecting walk on \(\{0,\ldots,L\}\), let \(T_{s,L}\) be the
time to hit zero from \(s\). For independent such walks, define

\[
p_{A,B}(s,t)=\Pr(T_{s,A}<T_{t,B})
\tag{4.1}
\]

away from simultaneous starts at zero, and assign the corner value \(1/2\).
For \(0\leq x\leq y\leq L\), set

\[
h_L(x,y)=2p_{L,L}(x,y)-1.
\tag{4.2}
\]

Its boundary conditions are

\[
h_L(0,y)=1,
\qquad
h_L(x,x)=0,
\tag{4.3}
\]

and, crucially, at the rate-one reflecting boundary,

\[
4h_L(x,L)
=h_L(x-1,L)+h_L(x+1,L)+2h_L(x,L-1).
\tag{4.4}
\]

The coefficient two in (4.4) comes from the unique endpoint neighbor having
jump probability one. A nearest-neighbor folding coupling gives the usual
strict monotonicity of hitting times and kernel comparisons.

\hypertarget{quadratic-barrier}{%
\subsection{Quadratic barrier}\label{quadratic-barrier}}

Define

\[
G_L(x,y)
=\frac{(L-x)^2-(L-y)^2}{L^2}.
\tag{4.5}
\]

The polynomial \(u^2-v^2\), with \(u=L-x\) and \(v=L-y\), has zero discrete
Laplacian on the unrestricted square lattice. It is even in \(v\), so the
replacement of the two neighbors \(v=\pm1\) by twice the neighbor \(v=1\)
preserves harmonicity at (4.4). It vanishes on the diagonal and satisfies

\[
G_L(0,y)=1-\frac{(L-y)^2}{L^2}\leq 1.
\]

The maximum principle therefore gives

\[
h_L(x,y)\geq G_L(x,y),
\qquad
h_L(L-d,L)\geq\frac{d^2}{L^2}.
\tag{4.6}
\]

For comparison, the lower barrier
\[
\frac{(y-x)(2L-x-y)}{2L(x+y)}
\]
has the same numerator as (4.5). Consequently, (4.5) improves this
comparison precisely when

\[
x+y>\frac L2,
\tag{4.7}
\]

not globally. For example, \(L=10,x=1,y=2\) gives \(17/100<17/60\).

\hypertarget{degree-six-improvement}{%
\subsection{Degree-six improvement}\label{degree-six-improvement}}

Let

\[
P_6(u,v)
=u^6-15u^4v^2+15u^2v^4-v^6-5(u^4-v^4).
\tag{4.8}
\]

The unrestricted discrete Laplacian satisfies

\[
\Delta\bigl(u^6-15u^4v^2+15u^2v^4-v^6\bigr)
=60(u^2-v^2),
\qquad
\Delta(u^4-v^4)=12(u^2-v^2),
\]

hence \(\Delta P_6=0\). The polynomial vanishes when \(u=v\) and is even in
\(v\), so it also satisfies the reflecting-boundary equation (4.4). For
\(L\geq3\), put

\[
D_L=P_6(L,0)=L^4(L^2-5)>0
\]

and

\[
\widetilde G_L(x,y)
=G_L(x,y)
+\frac1{30}
\left(
G_L(x,y)-\frac{P_6(L-x,L-y)}{D_L}
\right).
\tag{4.9}
\]

On the absorbing boundary \(x=0\), put \(v=L-y\). For \(v\geq1\),

\[
\frac{G_L(0,y)-P_6(L,v)/D_L}{v^2/L^2}
=\frac{14L^4+5L^2-(15L^2+5)v^2+v^4}
     {L^2(L^2-5)}
<30.
\tag{4.10}
\]

The numerator decreases as a function of \(v^2\in[1,L^2]\). The final
inequality follows from

\[
30L^2(L^2-5)-(14L^4-10L^2-4)
=4(4L^4-35L^2+1)>0.
\]

At \(v=0\), the correction vanishes. The maximum principle gives

\[
h_L(x,y)\geq\widetilde G_L(x,y).
\tag{4.11}
\]

In particular,

\[
h_L(L-d,L)
\geq\frac{d^2}{L^2}
+\frac{d^2(L^2-d^2)(L^2+d^2-5)}
       {30L^4(L^2-5)}
\geq\frac{d^2}{L^2}
+\frac1{30}
\left(
\frac{d^2}{L^2}-\frac{d^6}{L^6}
\right).
\tag{4.12}
\]

The last comparison reduces to the nonnegative numerator
\(5d^2(L^2-d^2)\).

\hypertarget{outer-flux-rectangle-domination-and-the-aggregate-theorem}{%
\section{Outer flux, rectangle domination, and the aggregate theorem}\label{outer-flux-rectangle-domination-and-the-aggregate-theorem}}

Write

\[
N=2m-\varepsilon,
\qquad
\varepsilon\in\{-1,0,1\},
\qquad
c=m,
\qquad
b=m-d.
\tag{5.1}
\]

Here \(\varepsilon=0\) is the odd path, \(\varepsilon=-1\) represents the lower
center of an even path, and \(\varepsilon=1\) represents its upper center.
Splitting target vertices into those outside and inside the interval between
the starts gives the exact identity

\[
2\left(
U_{P_{N+1}}(m,m-d)-\frac{N+1}{2}
\right)
=O_{m,d,\varepsilon}+I_{m,d,\varepsilon},
\tag{5.2}
\]

where

\[
O_{m,d,\varepsilon}
=\sum_{t=1}^{m-\varepsilon}h_{m+t}(t,t+d)
-\sum_{s=1}^{m-d}h_{m+d-\varepsilon+s}(s,s+d)
\tag{5.3}
\]

and

\[
I_{m,d,\varepsilon}
=-\sum_{r=1}^{d-1}\delta_{m,d,r}^{(\varepsilon)}.
\tag{5.4}
\]

With \(\ell=d-r\), the interior deficits are

\[
\delta_{m,d,r}^{(\varepsilon)}
=1-p_{m+r-\varepsilon,m-r}(r,\ell)
  -p_{m+\ell-\varepsilon,m-\ell}(\ell,r).
\tag{5.5}
\]

Equation (5.4) is an empty sum, equal to zero, when \(d=1\); no strict
inequality of the form \(0<0\) is used.

\hypertarget{outer-flux-comparison}{%
\subsection{Outer-flux comparison}\label{outer-flux-comparison}}

Set

\[
w_L(x,y)=h_{L+1}(x,y)-h_L(x,y)
\]

and

\[
s_L(x)=h_{L+1}(x,L+1)-h_{L+1}(x,L-1)>0.
\]

Strict positivity can also be seen by coupling the two second-coordinate
starts with the same underlying folded walk: the more distant start can only
increase the probability that the first coordinate reaches its target before
the walkers meet, and a prescribed finite sequence of jumps makes this
increase strict.

Subtracting the reflecting-boundary equations yields

\[
4w_L(x,L)
=w_L(x-1,L)+w_L(x+1,L)+2w_L(x,L-1)+s_L(x).
\tag{5.6}
\]

The maximum principle therefore gives \(w_L\geq0\). On the half triangle
\(x<y\) and \(x+y<L+1\), define

\[
\mathscr D_L(x,y)
=w_L(L+1-y,L+1-x)-w_L(x,y).
\tag{5.7}
\]

It is harmonic in the interior and vanishes on the diagonal and the
anti-diagonal. At its remaining boundary, direct substitution of (5.6) gives

\[
\begin{aligned}
4\mathscr D_L(1,y)
&=\mathscr D_L(1,y-1)
  +\mathscr D_L(1,y+1)
  +\mathscr D_L(2,y)\\
&\quad+w_L(L+1-y,L-1)+s_L(L+1-y).
\end{aligned}
\tag{5.8}
\]

The final two terms are nonnegative and the last is strictly positive.
Consequently \(\mathscr D_L(x,y)>0\) throughout its transient half triangle.
Telescoping the two sums in (5.3) now gives the exact identity

\[
\begin{aligned}
O_{m,d,\varepsilon}
&=\sum_{j=0}^{d-\varepsilon-1}
   h_{m+j+1}(m-d+j+1,m+j+1)\\
&\quad+
 \sum_{j=0}^{d-\varepsilon-1}
 \sum_{s=1}^{m-d}
 \mathscr D_{m+j+s}(s,s+d).
\end{aligned}
\tag{5.9}
\]

In particular,

\[
O_{m,d,\varepsilon}
\geq
\sum_{j=0}^{d-\varepsilon-1}
h_{m+j+1}(m-d+j+1,m+j+1).
\tag{5.10}
\]

Every parity convention thus contains at least the \(d-1\) terms with
\(L=m+1,\ldots,m+d-1\).

\hypertarget{common-dirichlet-rectangle}{%
\subsection{Common Dirichlet rectangle}\label{common-dirichlet-rectangle}}

Let

\[
a=m-d,
\qquad
\ell=d-r,
\qquad
W=2(a+r),
\qquad
H=2(a+\ell).
\tag{5.11}
\]

Let \(J_{m,d,r}\) be the probability that an ordinary nearest-neighbor walk in
the Dirichlet rectangle

\[
(0,W)\times(0,H)
\]

started at \((r,\ell)\) first exits through its right or top side.

The common folding construction represents the four relevant reflected
first-passage times on two independent underlying walks as

\[
R_0=T_{r,m-\ell},
\qquad
R_1=T_{r,m+r-\varepsilon},
\qquad
S_0=T_{\ell,m-r},
\qquad
S_1=T_{\ell,m+\ell-\varepsilon}.
\tag{5.12}
\]

The short intervals are contained in the long intervals: on either coordinate
the farther unfolded absorbing boundary is displaced by
\(d-\varepsilon\geq1\). Thus \(R_0\leq R_1\) and \(S_0\leq S_1\). The two
nonoverlap events are
\(\{R_1<S_0\}\) and \(\{S_1<R_0\}\). Hence

\[
\delta_{m,d,r}^{(\varepsilon)}
=\Pr\bigl([R_0,R_1]\cap[S_0,S_1]\neq\varnothing\bigr).
\tag{5.13}
\]

If the short-interval rectangle first exits through its left or bottom side,
one of the short and long exit times agrees, so these intervals cannot
overlap. Far-side exit is therefore necessary. Conversely, prescribe finitely
many successive jumps of the first unfolded coordinate toward its far
short-interval boundary and then its farther long-interval boundary while the
second coordinate makes no jump. Independence and exponential holding times
give this event strictly positive probability; it has far-side rectangle exit
and \(R_0<R_1<S_0\), hence no interval overlap. Therefore

\[
0\leq\delta_{m,d,r}^{(\varepsilon)}
<J_{m,d,r}.
\tag{5.14}
\]

Coordinatewise monotonicity and central reflection in the rectangle give

\[
J_{m,d,r}\leq\frac12,
\qquad
J_{m,d,r}=\frac12\ \Longleftrightarrow\ m=d.
\tag{5.15}
\]

The same strict comparison (5.14) applies to all three values of
\(\varepsilon\); it does not identify the inequivalent even-center rows.

\hypertarget{parameter-uniform-aggregate-rectangle-theorem}{%
\subsection{Parameter-uniform aggregate rectangle theorem}\label{parameter-uniform-aggregate-rectangle-theorem}}

\begin{theorem}[Aggregate rectangle exit]\label{thm:aggregate-rectangle}
For \(m>d\geq2\),

\[
\sum_{r=1}^{d-1}J_{m,d,r}
\leq
d^2\sum_{j=1}^{d-1}\frac1{(m+j)^2}.
\tag{5.16}
\]

In fact, with \(a=m-d\geq1\),

\[
\sum_{r=1}^{d-1}J_{m,d,r}
\leq
T_{a,d}
=\frac{d^2(d-1)}
       {(a+d+1)(a+2d-1)}.
\tag{5.17}
\]
\end{theorem}

\begin{proof}
The proof consists entirely of the following parameter-uniform estimates and
three explicitly handled exceptional pairs.

First consider the bounded quadrant barrier

\[
q(u,v)=\frac{3v}{4u}-\frac{v^3}{4u^3},
\qquad
u\geq1,\quad 0\leq v\leq u.
\tag{5.18}
\]

It is undefined at \((0,0)\), which is excluded. On its stated domain,
\(0\leq q\leq1/2\). For \(u\geq2\) and \(1\leq v<u\), direct calculation gives

\[
\Delta q(u,v)
=\frac{
v\left[
3(u^2-1)^2-v^2(6u^4-3u^2+1)
\right]
}{
2u^3(u^2-1)^3
}
<0.
\tag{5.19}
\]

Indeed, replacing \(v^2\) by its minimum value one bounds the numerator's
bracket above by

\[
-3u^4-3u^2+2<0.
\]

The stopped barrier is bounded. The wedge exit time is almost surely finite,
since its second coordinate hits zero almost surely. Applying optional
stopping to its minimum with a deterministic time, and then invoking bounded
convergence, therefore justifies optional stopping on the unbounded wedge.
Let

\[
\omega(x,y)
=\Pr_{(x,y)}\bigl(\tau_{\{X=0\}}<\tau_{\{Y=0\}}\bigr).
\]

Coordinate-exchange symmetry gives \(\omega(u,u)=1/2\), while
\(\omega(u,0)=0\). Optional stopping on the wedge \(0\leq y\leq x\) therefore
shows \(\omega(x,y)\leq q(x,y)\). If \(P_A(x,y)\) denotes the probability of
hitting \(x=2A\) before \(x=0\) or \(y=0\), reflection in the vertical line
\(x=A\) gives the exact midpoint identity

\[
\begin{aligned}
P_A(A,y)
&=\Pr_{(A,y)}
\bigl(
\tau_{\{X=0\}}
<
\tau_{\{X=2A\}}\wedge\tau_{\{Y=0\}}
\bigr)\\
&\leq\omega(A,y)
\leq q(A,y)
\leq\frac{3y}{4A}
\qquad(0\leq y\leq A).
\end{aligned}
\tag{5.19a}
\]

For \(y>A\), the same final bound is trivial once it reaches one, and in the
remaining range it follows by the same midpoint comparison using
\(P_A(A,y)\leq1/2\). The harmonic function \(3xy/(4A^2)\) and the half-strip
maximum principle now give

\[
P_A(x,y)\leq\frac{3xy}{4A^2},
\qquad
1\leq x\leq A,\quad y\geq1,
\tag{5.20}
\]

where \(P_A\) is as above. To justify the maximum principle on the
unbounded half-strip, truncate at height \(M\) with

\[
M\geq\left\lceil\frac{4A^2}{3}\right\rceil.
\tag{5.21}
\]

Then \(3xM/(4A^2)\geq1\) for every integer \(1\leq x\leq A\), so the top
boundary comparison is valid; taking \(M\to\infty\) proves (5.20).

Let \(P_R\) and \(P_T\) be the right-side and top-side exit probabilities of
the rectangle in (5.11). Applying (5.20) separately yields

\[
P_R\leq\frac{3r(d-r)}{4(a+r)^2},
\qquad
P_T\leq\frac{3r(d-r)}{4(a+d-r)^2}.
\tag{5.22}
\]

For a walk killed at both \(0\) and \(2A\), started at \(y\leq A\), reflect
each surviving path ending at \(y-j<y\) after its last visit to \(y\).
After that visit the original tail remains strictly below \(y\); its
reflection remains strictly above \(y\) and below \(2y\leq2A\). It therefore
survives, ends at \(y+j\), and has the same probability. This injection
pairs every lower endpoint with an equally weighted higher endpoint, so the
conditional mean of the surviving coordinate is at least \(y\). Conditioning
on the independent exit time of the other coordinate preserves the
comparison. The product of the two coordinate walks is a martingale, and its
stopped value is bounded by \(WH\); bounded convergence at the almost surely
finite rectangle exit time therefore gives

\[
\begin{aligned}
r(d-r)
&=2(a+r)\mathbb E[Y_\tau;\text{right exit}]\\
&\quad+2(a+d-r)\mathbb E[X_\tau;\text{top exit}]\\
&\geq2(a+r)(d-r)P_R+2(a+d-r)rP_T.
\end{aligned}
\tag{5.23}
\]

The certificate checks this killed-walk reflection comparison independently
for 8,835 exact instances, but the reflection injection, not the finite
checks, is the parameter-uniform justification.

Put

\[
s=\min(r,d-r),
\qquad
t=\max(r,d-r).
\]

Combining the product constraint with the stronger of the individual bounds
in (5.22) gives

\[
J_{m,d,r}=P_R+P_T
\leq C_s
=\frac{s}{2(a+s)}
+\frac{3as(t-s)}
       {4(a+s)(a+t)^2}.
\tag{5.24}
\]

For example, if \(r=s\) and \(d-r=t\), solve (5.23) for \(P_R\) to obtain

\[
P_R+P_T
\leq
\frac{s}{2(a+s)}
+\frac{a(t-s)}{t(a+s)}P_T;
\]

substitution of the top-side bound from (5.22) yields (5.24). The opposite
ordering follows by coordinate exchange.

There are two parameter ranges.

\textbf{Range I: \(a\geq d\).} Except at \((a,d)=(2,2)\), convexity and symmetry in
\(r\) give

\[
\frac1{(a+r)^2}+\frac1{(a+d-r)^2}
\leq
\frac{8d}
{(d+1)(a+d+1)(a+2d-1)}.
\tag{5.25}
\]

It is enough to check \(r=1\). Set \(p=a-d\geq0\) and \(t=d-2\geq0\).
After clearing positive denominators, the required difference is

\[
\begin{aligned}
&-54+90t+219t^2+130t^3+29t^4+2t^5\\
&\quad+p(288+816t+738t^2+271t^3+35t^4)\\
&\quad+p^2(300+544t+311t^2+57t^3)\\
&\quad+p^3(96+112t+32t^2)+p^4(10+6t).
\end{aligned}
\tag{5.26}
\]

If \(t\geq1\), the \(90t\) term dominates the sole negative constant. If
\(p\geq1\), the \(288p\) term does so. The only omitted pair is \(p=t=0\).
Summing (5.22), using (5.25) and

\[
\sum_{r=1}^{d-1}r(d-r)=\frac{d(d^2-1)}6,
\]

proves (5.17).

\textbf{Range II: \(1\leq a\leq d\).} Elementary maximization gives

\[
\frac{s(t-s)}{(a+t)^2}
\leq
\frac{d^2}{4(a+d)(2a+d)}.
\tag{5.27}
\]

Indeed the function

\[
s\longmapsto\frac{s(d-2s)}{(a+d-s)^2},
\qquad
0\leq s\leq\frac d2,
\]

has its maximum at

\[
s_*=\frac{d(a+d)}{4a+3d},
\]

and direct substitution gives the right-hand side of (5.27).

Define

\[
\kappa_{a,d}
=\frac{3d^2}{8(a+d)(2a+d)}.
\tag{5.28}
\]

Substitution into (5.24) yields the inequality

\[
\boxed{
\sum_{r=1}^{d-1}J_{m,d,r}
\leq\frac{d-1}{2}
-(1-\kappa_{a,d})\frac a2
\sum_{r=1}^{d-1}
\frac1{a+\min(r,d-r)}.
}
\tag{5.29}
\]

Jensen's inequality and

\[
\sum_{r=1}^{d-1}\min(r,d-r)
=\left\lfloor\frac{d^2}{4}\right\rfloor
\leq\frac{d^2}{4}
\]

imply

\[
\sum_{r=1}^{d-1}
\frac1{a+\min(r,d-r)}
\geq
\frac{d-1}{a+d^2/[4(d-1)]}.
\tag{5.30}
\]

The saving required to reach (5.17) is

\[
\frac{d-1}{2}-T_{a,d}
=\frac{(d-1)(a^2+3ad+d-1)}
       {2(a+d+1)(a+2d-1)}.
\tag{5.31}
\]

After clearing positive denominators, (5.29)--(5.31) reduce to the sign of

\[
\begin{aligned}
\mathcal N(a,d)
=d^2\bigl[
&-4a^4+(11d-29)a^3+(19d^2-43d+4)a^2\\
&+(4d^3-19d^2+12d-3)a-2d^3+2d^2
\bigr].
\end{aligned}
\tag{5.32}
\]

For \(a=1,d=t+4\), the bracket becomes

\[
48+92t+26t^2+2t^3>0.
\tag{5.33}
\]

For \(a=2,d=t+3\), it becomes

\[
56+342t+94t^2+6t^3>0.
\tag{5.34}
\]

For \(a=p+3,d=a+s\), it becomes

\[
\begin{aligned}
&72+912s+206s^2+10s^3\\
&\quad+p(834+1141s+161s^2+4s^3)\\
&\quad+p^2(801+462s+31s^2)
       +p^3(267+61s)+30p^4>0.
\end{aligned}
\tag{5.35}
\]

The three exceptional parameter pairs are handled directly:

\[
(a,d)=(1,2):
\qquad
\sum C_s=T_{1,2}=\frac14;
\tag{5.36}
\]

\[
(a,d)=(1,3):
\qquad
\mathcal N(1,3)=-180,
\qquad
\sum C_s=\frac7{12}<\frac35=T_{1,3};
\tag{5.37}
\]

\[
(a,d)=(2,2):
\qquad
J_{4,2,1}=\frac{31}{495}<\frac4{25}=T_{2,2}.
\tag{5.38}
\]

Only the last pair requires an exceptional explicitly solved Dirichlet
rectangle: its value follows from all 25 interior harmonic equations of the
\(6\times6\) rectangle.

Finally,

\[
\begin{aligned}
\sum_{j=1}^{d-1}\frac1{(m+j)^2}
&\geq
\frac{m}{m+1}
\sum_{j=1}^{d-1}
\frac1{(m+j-1)(m+j)}\\
&=\frac{d-1}{(m+1)(m+d-1)}.
\end{aligned}
\tag{5.39}
\]

Because \(m=a+d\), the right-hand side times \(d^2\) equals \(T_{a,d}\).
This proves (5.16) and completes the all-parameter argument.
\end{proof}

\hypertarget{complete-finite-path-pure-and-mixed-equilibria}{%
\section{Complete finite-path pure and mixed equilibria}\label{complete-finite-path-pure-and-mixed-equilibria}}

\begin{theorem}[Complete finite-path equilibrium classification]
\label{thm:path-equilibria}
For every finite path \(P_n\), the safe vertices are exactly its central
vertices, every central vertex strictly defeats every noncentral vertex, and
\[
\mathcal O(P_n)=\mathcal P(C_{P_n}),\qquad
\operatorname{PNE}(P_n)=C_{P_n}\times C_{P_n},\qquad
\operatorname{MNE}(P_n)
=\mathcal P(C_{P_n})\times\mathcal P(C_{P_n}).
\]
In particular, the odd path has one safe vertex and the even path has two.
The degenerate paths \(P_1\) and \(P_2\) are included.
\end{theorem}

\begin{proof}
For \(2\leq d\leq m-1\), (5.10), (4.12), (5.14), and (5.16) give

\[
\begin{aligned}
2\left(
U_{P_{N+1}}(m,m-d)-\frac{N+1}{2}
\right)
&>
\frac1{30}\Phi_m(d),\\
\Phi_m(d)
&=
\sum_{j=1}^{d-1}
\left(
\frac{d^2}{(m+j)^2}
-\frac{d^6}{(m+j)^6}
\right)
>0.
\end{aligned}
\tag{6.1}
\]

Every term of \(\Phi_m(d)\) has the displayed subtraction sign and is
positive. The strictness follows from each strict rectangle inclusion (5.14);
no pointwise conjectural rectangle inequality is needed.

When \(d=1\), the interior sum is empty. For odd paths and lower even centers
the outer sum is strictly positive. For an upper even center, the opponent at
distance one is the other center, and symmetry gives payoff exactly
\((N+1)/2\). In particular,

\[
U_{P_3}(1,0)-\frac32=\frac18>0.
\tag{6.2}
\]

For endpoint opponents, \(d=m\). If \(\varepsilon=0\), (4.6) and telescoping
give

\[
O_{m,m,0}
\geq
m^2\sum_{L=m+1}^{2m}\frac1{L^2}
\geq
\frac{m^2}{2(m+1)}
=\frac{m-1}{2}+\frac1{2(m+1)}.
\tag{6.3}
\]

For \(\varepsilon=-1\), one additional positive outer term is available.
These arguments use

\[
-I_{m,m,\varepsilon}<\frac{m-1}{2}
\tag{6.4}
\]

only when \(m\geq2\); the case \(m=d=1\) is treated directly by (6.2).
For the lower even center with \(m=d=1\), use the \(P_4\) value in (6.8);
the upper-center case \(m=d=1\) is \(P_2\), whose two vertices are both
central.

For the upper even center, \(\varepsilon=1\), convexity and the composite
trapezoid estimate yield

\[
\sum_{L=m+1}^{2m-1}\frac{m^2}{L^2}
\geq
\frac m2-\frac58.
\tag{6.5}
\]

The degree-six correction therefore gives

\[
2\left(
U_{P_{2m}}(m,0)-m
\right)
>
\frac1{30}\Phi_m(m)-\frac18.
\tag{6.6}
\]

For \(m\geq24\), at least 16 indices satisfy \(4m/3\leq L<2m\), and

\[
\frac{m^2}{L^2}-\frac{m^6}{L^6}
>\frac{15}{64};
\]

hence the correction exceeds \(1/8\). The cases \(12\leq m\leq23\) follow
from the following stronger bridge, rather than from the possibly negative
right-hand side of (6.6). Define

\[
\begin{aligned}
B_m
&=
\sum_{L=m+1}^{2m-1}\frac{m^2}{L^2}
+\frac1{30}
\sum_{L=m+1}^{2m-1}
\left(
\frac{m^2}{L^2}-\frac{m^6}{L^6}
\right)
-\frac{m-1}{2}.
\end{aligned}
\tag{6.6a}
\]

For \(\varepsilon=1\), the first sum in (5.10) contains exactly the indices
\(L=m+1,\ldots,2m-1\). Applying the degree-six barrier (4.12) to each of those
terms and using

\[
-I_{m,m,1}<\frac{m-1}{2}
\]

in (5.2) gives the strict estimate

\[
2\left(U_{P_{2m}}(m,0)-m\right)>B_m.
\tag{6.6b}
\]

Every \(B_m\) is a finite rational number. The supplied finite-path certificate
constructs each summand with Python's exact
\(\texttt{fractions.Fraction}\), subtracts \((m-1)/2\), and executes
\(\texttt{assert margin > 0}\) separately for each
\(m=12,13,\ldots,23\). It then prints all twelve reduced fractions. Their
minimum is attained at \(m=12\), where

\[
B_{12}
=\frac{
45099969471864349286182125482800145174460998792423
}{
9468373108857039328547559427556736346876985011200000
}>0.
\tag{6.7}
\]

The remaining \(2\leq m\leq11\) follow from exact product-chain calculations
for both even centers. For each \(m\), the same supplied certificate constructs
the complete transient product-chain matrix with the endpoint transition
weight doubled, solves it by fraction-free integer elimination, substitutes
the exact rational solution back into every harmonic equation, and asserts
positivity of

\[
U_{P_{2m}}(m,0)-m
\quad\text{and}\quad
U_{P_{2m}}(m-1,0)-m.
\tag{6.7a}
\]

It prints the complete twenty-entry table for \(2\leq m\leq11\). At \(P_4\),
the inequivalent left-endpoint advantages are

\[
U_{P_4}(2,0)-2=\frac3{182},
\qquad
U_{P_4}(1,0)-2=\frac{19}{104}.
\tag{6.8}
\]

Reflection supplies the right-endpoint cases after the two left-center rows
have been handled separately. Thus every path center strictly defeats every
noncentral opponent.

Every noncentral vertex \(b\) is excluded from \(C_{P_n}\) without requiring a
separate adjacent-sign theorem: take any center \(c\), for which the strict
victory just proved gives

\[
A_{P_n}(b,c)=-A_{P_n}(c,b)<0.
\]

Conversely all central vertices are safe, including their zero-payoff
matchups against one another. Therefore

\[
C_{P_{2m+1}}=\{m\},
\qquad
C_{P_{2m}}=\{m-1,m\}.
\tag{6.9}
\]

For \(P_1\), the unique center is \(0\). For \(P_2\), both vertices are
central and their mutual payoff is one.

To classify optimal mixed strategies, fix any center \(c\). If
\(\mu\in\mathcal O(P_n)\), then

\[
0
\leq
\sum_{b\in V(P_n)}
\mu(b)
\left(
U_{P_n}(b,c)-\frac n2
\right).
\tag{6.10}
\]

Every summand for a noncentral \(b\) is strictly negative, while every central
summand is zero. Hence \(\operatorname{supp}(\mu)\subseteq C_{P_n}\).
Conversely, every measure supported on the safe centers is optimal. Thus

\[
\mathcal O(P_n)=\mathcal P(C_{P_n}),
\tag{6.11}
\]

or, explicitly,

\[
\mathcal O(P_{2m+1})=\{\delta_m\},
\tag{6.12}
\]

\[
\mathcal O(P_{2m})
=\left\{
p\delta_{m-1}+(1-p)\delta_m:
0\leq p\leq1
\right\}.
\tag{6.13}
\]

Equations (1.6) and (6.9)--(6.13) give the complete pure and mixed
classifications, including \(P_1\) and \(P_2\):

\[
\operatorname{PNE}(P_{2m+1})=\{(m,m)\},
\qquad
\operatorname{PNE}(P_{2m})=\{m-1,m\}^2,
\tag{6.14}
\]

\[
\operatorname{MNE}(P_n)
=\mathcal P(C_{P_n})\times\mathcal P(C_{P_n}).
\tag{6.15}
\]
\end{proof}

\hypertarget{uniform-reflected-brownian-limit-and-cubic-equilibrium-stability}{%
\section{Uniform reflected-Brownian limit and cubic equilibrium stability}\label{uniform-reflected-brownian-limit-and-cubic-equilibrium-stability}}

Let \(B^{(x)}\) be reflected Brownian motion on \([0,1]\) with generator
\(\frac12\partial_{xx}\), started at \(x\), and let

\[
\sigma_z^{(x)}
=\inf\{t\geq0:B_t^{(x)}=z\}.
\]

For independent motions, define

\[
Q_z(x,y)
=\Pr\bigl(\sigma_z^{(x)}<\sigma_z^{(y)}\bigr)
+\frac12
 \Pr\bigl(\sigma_z^{(x)}=\sigma_z^{(y)}\bigr)
\tag{7.1}
\]

and the normalized continuum payoff

\[
F(x,y)=\int_0^1Q_z(x,y)\,dz.
\tag{7.2}
\]

\begin{theorem}[Uniform Brownian limit and cubic midpoint stability]
\label{thm:brownian-limit}
The payoff \(F\) is jointly continuous on \([0,1]^2\), and
\[
\sup_{0\leq a,b\leq N}
\left|
\frac{U_{P_{N+1}}(a,b)}{N+1}
-F\!\left(\frac aN,\frac bN\right)
\right|\longrightarrow0.
\]
Moreover, for every \(y\in[0,1]\),
\[
F(1/2,y)-1/2\geq |1-2y|^3/400.
\]
Consequently \(\delta_{1/2}\) is the unique optimal continuum strategy and
\((\delta_{1/2},\delta_{1/2})\) is the unique mixed equilibrium.
\end{theorem}

The following first-passage lemma and the ensuing estimates prove the theorem.

\hypertarget{varying-start-and-varying-target-first-passage-lemma}{%
\subsection{Varying-start and varying-target first-passage lemma}\label{varying-start-and-varying-target-first-passage-lemma}}

Define the triangular folding map

\[
\rho(u)=\min_{j\in2\mathbb Z}|u-j|,
\qquad
u\in\mathbb R.
\tag{7.3a}
\]

It is continuous, one-Lipschitz, \(2\)-periodic, and takes values in
\([0,1]\). If \(S\) is a rate-one continuous-time simple symmetric walk on
\(\mathbb Z\), then

\[
X_t^{(N,a)}
=N\rho\left(\frac{a+S_t}{N}\right)
\tag{7.3b}
\]

is exactly the constant-speed reflecting walk on
\(\{0,\ldots,N\}\) started from \(a\). At an interior vertex the two
unfolded increments give the two neighbors with probability \(1/2\). At
\(0\) or \(N\), both unfolded increments fold onto the unique neighbor, so
the total endpoint jump rate remains one.

\begin{lemma}[Varying starts and targets]\label{lem:varying-passage}
Suppose

\[
\frac{a_N}{N}\longrightarrow x,
\qquad
\frac{b_N}{N}\longrightarrow y,
\qquad
\frac{z_N}{N}\longrightarrow z\notin\{x,y\}.
\tag{7.3c}
\]

Then

\[
q_{z_N}^{(N)}(a_N,b_N)
\longrightarrow
Q_z(x,y).
\tag{7.3d}
\]

This remains valid when \(x=y\), and when \(z=0\) or \(z=1\).
\end{lemma}

\begin{proof}
Let \(S^{(1)}\) and \(S^{(2)}\) be independent rate-one simple
walks. The two-dimensional invariance principle gives

\[
\left(
\frac{a_N+S^{(1)}_{N^2t}}{N},
\frac{b_N+S^{(2)}_{N^2t}}{N}
\right)_{t\geq0}
\Longrightarrow
\left(x+W_t^{(1)},y+W_t^{(2)}\right)_{t\geq0},
\tag{7.3e}
\]

where \(W^{(1)},W^{(2)}\) are independent standard Brownian motions.
Continuity of \(\rho\) and (7.3b) therefore imply convergence, uniformly on
compact time intervals in the continuous-limit topology, to two independent
reflected Brownian motions

\[
\rho(x+W^{(1)}),
\qquad
\rho(y+W^{(2)}).
\tag{7.3f}
\]

This also proves independence when \(x=y\): the initial positions may
coincide, but the unfolded increments and limiting Brownian motions remain
independent.

We record the hitting-time continuity used next. If continuous paths
\(\omega_n\) converge uniformly on compact intervals to \(\omega\), levels
\(c_n\to c\), and \(\omega\) crosses \(c\) immediately after its first hit,
then the first hitting times of \(c_n\) converge to the first hitting time of
\(c\). For the linearly interpolated nearest-neighbor walks, crossing a
lattice level forces an exact visit to that level. The maximal interpolation
error on a fixed diffusive time interval tends to zero in probability, so the
same statement applies to (7.3e).

For an interior reflected target \(z\in(0,1)\), the unfolded target set is

\[
\rho^{-1}(z)=\{2j+z,2j-z:j\in\mathbb Z\}.
\tag{7.3g}
\]

At its first hit, Brownian motion almost surely visits both sides of the hit
level in every subsequent time interval. Thus the required crossing
condition holds almost surely. At \(z=0\) and \(z=1\), the reflected path
cannot cross the endpoint, but the unfolded Brownian path crosses,
respectively, an even or odd integer immediately after first contact. Applying
the same argument to the unfolded target set proves endpoint-target
continuity as well.

Because \(z\notin\{x,y\}\), both limiting first-passage times are strictly
positive and finite almost surely. Moreover, the preimages in (7.3g), or the
even/odd integer preimages at an endpoint, bound the lifted starting point in
a finite interval. The reflected hitting time is therefore the exit time of
ordinary Brownian motion from that interval. Its law has a continuous
density. The two limiting hitting times are independent and atomless, so

\[
\Pr\bigl(\sigma_z^{(x)}=\sigma_z^{(y)}\bigr)=0.
\tag{7.3h}
\]

The continuous-mapping theorem applied jointly to the two passage times,
followed by (7.3h), proves convergence of the race indicator. First truncate
at a deterministic time \(T\); then let \(T\to\infty\), using almost-sure
finiteness of both limiting passage times. This proves (7.3d), including
coincident limiting starts and endpoint targets.
\end{proof}

The same coupling proves a continuum continuity statement that will also be
needed below:

\[
(x_n,y_n,z_n)\to(x,y,z),
\qquad
z\notin\{x,y\}
\quad\Longrightarrow\quad
Q_{z_n}(x_n,y_n)\to Q_z(x,y).
\tag{7.3i}
\]

Indeed, use common independent Brownian increments in (7.3f) and repeat the
almost-sure hitting-time argument. Thus \(Q\) is jointly continuous on

\[
\mathcal D
=\{(x,y,z)\in[0,1]^3:z\notin\{x,y\}\}.
\tag{7.3j}
\]

\hypertarget{joint-continuity-of-the-integrated-payoff}{%
\subsection{Joint continuity of the integrated payoff}\label{joint-continuity-of-the-integrated-payoff}}

Let \((x_n,y_n)\to(x,y)\). For every
\(z\notin\{x,y\}\), (7.3i) gives

\[
Q_z(x_n,y_n)\longrightarrow Q_z(x,y).
\]

The exceptional target set has Lebesgue measure zero and \(0\leq Q\leq1\).
Dominated convergence in (7.2) therefore proves

\[
F(x_n,y_n)\longrightarrow F(x,y).
\tag{7.3k}
\]

Hence \(F\) is jointly continuous on the whole square \([0,1]^2\), including
the diagonal \(x=y\) and boundary starting positions.

\hypertarget{uniform-two-start-convergence}{%
\subsection{Uniform two-start convergence}\label{uniform-two-start-convergence}}

We first prove convergence along arbitrary convergent start sequences. Suppose

\[
\frac{a_N}{N}\to x,
\qquad
\frac{b_N}{N}\to y.
\tag{7.3l}
\]

Fix \(\eta>0\), and remove the target neighborhoods

\[
\mathcal V_\eta
=\bigl((x-\eta,x+\eta)\cup(y-\eta,y+\eta)\bigr)\cap[0,1].
\tag{7.3m}
\]

On the compact complement of \(\mathcal V_\eta\),
Lemma~\ref{lem:varying-passage} and the joint
continuity (7.3i) are uniform over target grid points. Otherwise there would
be a violating sequence \(z_N/N\) in that compact set; a convergent
subsequence would contradict (7.3d) and (7.3i). Consequently

\[
\sup_{\substack{0\leq j\leq N\\j/N\notin\mathcal V_\eta}}
\left|
q_j^{(N)}(a_N,b_N)-Q_{j/N}(x,y)
\right|
\longrightarrow0.
\tag{7.3n}
\]

Both integrands lie in \([0,1]\), while the proportion of target grid points
inside \(\mathcal V_\eta\) is at most

\[
4\eta+O(N^{-1}).
\tag{7.3o}
\]

On the complement, \(z\mapsto Q_z(x,y)\) is continuous, so its grid average
is a Riemann sum. Equations (7.3n)--(7.3o), followed by
\(\eta\downarrow0\), give

\[
\frac{U_{P_{N+1}}(a_N,b_N)}{N+1}
\longrightarrow F(x,y).
\tag{7.3p}
\]

The two endpoint targets could alternatively be discarded at total cost at
most \(2/(N+1)\); Lemma~\ref{lem:varying-passage} already treats them
directly.

Now suppose the desired uniform convergence failed. There would exist
\(\epsilon>0\), integers \(N_k\to\infty\), and starts \(a_k,b_k\) such that

\[
\left|
\frac{U_{P_{N_k+1}}(a_k,b_k)}{N_k+1}
-F\left(\frac{a_k}{N_k},\frac{b_k}{N_k}\right)
\right|
\geq\epsilon.
\tag{7.3q}
\]

Compactness supplies a subsequence on which
\(a_k/N_k\to x\) and \(b_k/N_k\to y\). Along that subsequence, (7.3p)
sends the first term to \(F(x,y)\), while joint continuity (7.3k) sends the
second term to the same limit. This contradicts (7.3q). Therefore

\[
\boxed{
\lim_{N\to\infty}
\sup_{0\leq a,b\leq N}
\left|
\frac{U_{P_{N+1}}(a,b)}{N+1}
-F\left(\frac aN,\frac bN\right)
\right|
=0.
}
\tag{7.3}
\]

No convergence rate is claimed.

For completeness, the product Dirichlet problem has the explicit separated
series

\[
Q_z(x,y)
=\sum_{j=0}^{\infty}
\frac{2(-1)^j}{\beta_j}
\cos\left(\frac{\beta_j w}{L_y}\right)
\frac{\cosh(\beta_j u/L_y)}
     {\cosh(\beta_j L_x/L_y)},
\qquad
\beta_j=\left(j+\frac12\right)\pi,
\tag{7.4}
\]

when the distances from the reflecting endpoints and the reflecting-side
lengths of the two starting positions are \((u,L_x)\) and \((w,L_y)\),
respectively. For a starting
position \(s<z\), these are

\[
(u,L)=(s,z);
\]

for \(s>z\), they are

\[
(u,L)=(1-s,1-z).
\]

The zero-target boundary in the first coordinate and reflecting far boundaries
produce the half-integer modes in (7.4).

Let

\[
\theta=|1-2y|\in[0,1].
\]

For interior opponents, take sequences of path centers and opponent vertices
whose scaled locations converge to \(1/2\) and \(y\). The degree-six improvement
(6.1) and a Riemann-sum limit give

\[
F\left(\frac12,y\right)-\frac12
\geq\frac{\Psi(\theta)}{120},
\tag{7.5}
\]

where

\[
\begin{aligned}
\Psi(\theta)
&=
\int_0^\theta
\left(
\frac{\theta^2}{(1+t)^2}
-\frac{\theta^6}{(1+t)^6}
\right)\,dt\\
&=
\frac{\theta^3}{1+\theta}
-\frac{\theta^6}{5}
\left(
1-\frac1{(1+\theta)^5}
\right).
\end{aligned}
\tag{7.6}
\]

At \(y=0,1\), use (6.6) instead; its fixed correction \(1/8\) vanishes after
normalization and passage to the limit. In particular,

\[
\Psi(1)=\frac{49}{160}.
\tag{7.7}
\]

For \(0\leq\theta\leq1\),

\[
\frac{\theta^3}{1+\theta}\geq\frac{\theta^3}{2},
\qquad
\frac{\theta^6}{5}
\left(1-\frac1{(1+\theta)^5}\right)
\leq\frac{\theta^3}{5},
\]

and hence

\[
\boxed{
F\left(\frac12,y\right)-\frac12
\geq
\frac{|1-2y|^3}{400}.
}
\tag{7.8}
\]

The inequality is strict whenever \(y\neq1/2\). Skew symmetry gives

\[
F(x,1/2)-\frac12
\leq-\frac{|1-2x|^3}{400}.
\tag{7.9}
\]

Testing any optimal probability measure against the pure midpoint therefore
forces it to be concentrated at \(1/2\). Conversely the midpoint is safe by
(7.8). The unique continuum optimal strategy and mixed Nash equilibrium are

\[
\mathcal O_{\mathrm{cont}}=\{\delta_{1/2}\},
\qquad
\operatorname{MNE}_{\mathrm{cont}}
=\{(\delta_{1/2},\delta_{1/2})\}.
\tag{7.10}
\]

If \(\mu\) is \(\eta\)-optimal in normalized payoff, (7.9) implies

\[
\int_0^1|1-2x|^3\,d\mu(x)\leq400\eta
\tag{7.11}
\]

and therefore

\[
\mu\bigl(\{|x-1/2|\geq\rho\}\bigr)
\leq\frac{50\eta}{\rho^3}.
\tag{7.12}
\]

For the ordinary unilateral-deviation definition of an additive
\(\varepsilon\)-Nash equilibrium, each marginal is \(2\varepsilon\)-optimal,
so the corresponding constants are \(800\) and \(100\).

\hypertarget{unique-continuum-endpoint-optimizer}{%
\section{Unique continuum endpoint optimizer}\label{unique-continuum-endpoint-optimizer}}

\begin{theorem}[Unique endpoint response and certified localization]
\label{thm:endpoint-optimizer}
The continuum endpoint payoff \(H(x)=F(x,0)\) has a unique global maximizer
\(x_*\), satisfying
\[
0.363424228<x_*<0.363424229,
\qquad
0.57155664213<H(x_*)<0.57155664221.
\]
If \(a_N\) is any endpoint best response on \(P_{N+1}\), then
\(a_N/N\to x_*\).
\end{theorem}

The proof combines global interval localization, an image-series
representation valid across the moving diagonal, rigorous tail bounds, and
strict concavity on the localization interval.

Define

\[
H(x)=F(x,0),
\qquad
0\leq x\leq1.
\tag{8.1}
\]

Splitting target locations at \(z=x\) gives

\[
H(x)=L(x)+R(x),
\qquad
R(x)
=\int_x^1g(x/z)\,dz
=x\int_x^1\frac{g(r)}{r^2}\,dr,
\tag{8.2}
\]

where

\[
g(r)
=\sum_{j=0}^{\infty}
\frac{2(-1)^j}{\beta_j}
\frac{\cosh(\beta_j r)}{\cosh\beta_j},
\qquad
\beta_j=\left(j+\frac12\right)\pi.
\tag{8.3}
\]

\hypertarget{global-localization-and-alternating-series-validity}{%
\subsection{Global localization and alternating-series validity}\label{global-localization-and-alternating-series-validity}}

For \(0\leq x\leq1\), the positive alternating-series terms are decreasing in
their spectral parameter. Indeed, with

\[
b_\beta(x)
=\frac{\cosh(\beta x)}{\beta\cosh\beta},
\]

direct differentiation gives

\[
\frac{\partial}{\partial\beta}\log b_\beta(x)
=x\tanh(\beta x)-\tanh\beta-\frac1\beta
<0.
\tag{8.4}
\]

The strict inequality follows from \(x\leq1\), monotonicity of \(\tanh\), and
\(\beta>0\). This supplies the elementary justification required for every
alternating enclosure of \(g\).

For a mesh cell \([a,b]\), first-passage monotonicity on target positions to the
left of \(a\) and to the right of \(b\), together with the trivial bound one
on the intervening target interval, gives

\[
\sup_{x\in[a,b]}H(x)
\leq
L(a)+(b-a)+R(b).
\tag{8.5}
\]

The independent global-localization certificate uses outward rational
enclosures of (8.3) and (8.5) to prove

\[
H(3/8)>0.570193446836186350,
\tag{8.6}
\]

\[
\sup_{0\leq x\leq1}H(x)<0.576578462545287353,
\tag{8.7}
\]

\[
\sup_{x\notin[1/4,1/2]}H(x)
<0.569147328628688909.
\tag{8.8}
\]

Consequently

\[
0.57<\max_{0\leq x\leq1}H(x)<0.58,
\qquad
\operatorname{argmax}H\subset(1/4,1/2).
\tag{8.9}
\]

\hypertarget{image-series-and-the-moving-diagonal}{%
\subsection{Image series and the moving diagonal}\label{image-series-and-the-moving-diagonal}}

Write

\[
f(u)=\operatorname{sech}\left(\frac{\pi u}{2}\right),
\qquad
\chi(u)=\frac4\pi
\arctan\left(e^{-\pi u/2}\right),
\qquad
\chi'(u)=-f(u).
\tag{8.10}
\]

For \(x\in[1/4,1/2]\), set

\[
b=\frac{1-x}{x}\geq1,
\qquad
t=\frac{x}{z}-1
\quad\text{for }0<z<x,
\tag{8.11}
\]

and

\[
A_\ell=2\ell b,
\qquad
C_{\ell,\pm}=2\ell(b+1)\pm1=\frac{2\ell}{x}\pm1.
\tag{8.12}
\]

Expanding the hyperbolic denominators geometrically in (7.4), integrating the
resulting exponentials, and grouping adjacent images gives

\[
Q_{x/(1+t)}(x,0)
=\chi(t)
+\sum_{\ell=1}^{\infty}
(-1)^{\ell-1}
\left[
\chi(A_\ell+C_{\ell,-}t)
-\chi(A_\ell+C_{\ell,+}t)
\right].
\tag{8.13}
\]

Here the derivation can be made explicit before any interchange at the
diagonal. For \(z=x/(1+t)<x\), the two reflecting-side ratios in (7.4) are
\[
P=b(1+t),
\qquad
D=b+(b+1)t,
\qquad
D-P=t.
\]
For \(t>0\), the absolutely convergent geometric expansion is
\[
\frac{\cosh(\beta P)}{\cosh(\beta D)}
=\sum_{n=0}^{\infty}(-1)^n
\left(e^{-\beta((2n+1)D-P)}
     +e^{-\beta((2n+1)D+P)}\right).
\]
The first term is \(e^{-\beta t}\). Pairing the second term of index
\(n=\ell-1\) with the first term of index \(n=\ell\) gives
\[
(-1)^{\ell-1}
\left(
e^{-\beta(A_\ell+C_{\ell,-}t)}
-e^{-\beta(A_\ell+C_{\ell,+}t)}
\right).
\]
Finally,
\[
\sum_{j=0}^{\infty}\frac{2(-1)^j}{\beta_j}e^{-\beta_j u}
=\frac4\pi\arctan\!\left(e^{-\pi u/2}\right)
=\chi(u),
\qquad u>0,
\]
which proves (8.13) for \(t>0\). At \(t=0\), the paired image brackets
vanish and the leading alternating series has its Abel-continuous limit
\(\chi(0)=1\).

At the diagonal \(t=0\), every bracket vanishes and \(\chi(0)=1\).
Consequently (8.13) handles the boundary layer at the moving target \(z=x\)
without differentiating a nonuniform spectral series term by term.

For \(q=0,1,2\), differentiation of the image arguments and exponential decay
give constants \(K_q<\infty\), independent of \(\ell,b,t\), such that

\[
\left|
\partial_b^q
\chi(A_\ell+C_{\ell,\pm}t)
\right|
\leq
K_q\ell^q(1+t)^q
\exp\left[
-\frac{\pi}{2}
(A_\ell+C_{\ell,\pm}t)
\right].
\tag{8.14}
\]

Since \(b\geq1\),

\[
A_\ell\geq2\ell,
\qquad
C_{\ell,-}\geq4\ell-1.
\tag{8.15}
\]

Thus the right-hand side of (8.14), divided by the Jacobian factor
\((1+t)^2\), is summable in \(\ell\), integrable in \(t\geq0\), and uniformly
dominated for \(x\in[1/4,1/2]\). Dominated convergence justifies the image
rearrangement, the first two derivatives, and the interchange of summation
with the target integral, including at \(t=0\).

Define

\[
B_0=\int_0^\infty\frac{f(u)}{u+1}\,du,
\qquad
J(x)=\int_x^1\frac{g'(r)}{r}\,dr,
\tag{8.16}
\]

\[
J_{\ell,\pm}(x)
=\int_{A_\ell}^{\infty}
\frac{f(u)}{u+2\ell\pm1}\,du,
\tag{8.17}
\]

and

\[
V_\ell(x)
=C_{\ell,+}J_{\ell,+}(x)
-C_{\ell,-}J_{\ell,-}(x),
\qquad
E_\ell(x)=J_{\ell,-}(x)+J_{\ell,+}(x).
\tag{8.18}
\]

For \(A,C>0\), integration by parts gives

\[
\int_0^\infty\frac{\chi(A+Ct)}{(1+t)^2}\,dt
=\chi(A)
-C\int_A^\infty
\frac{f(u)}{u+C-A}\,du.
\tag{8.19}
\]

For the two images in (8.13), \(C_{\ell,\pm}-A_\ell=2\ell\pm1\).
Their equal \(\chi(A_\ell)\) terms cancel in (8.19), leaving precisely the
difference \(V_\ell\) in (8.18). In particular, no uncanceled moving-diagonal
boundary term remains. Since
\[
J_{\ell,\pm}'(x)
=\frac{2\ell}{x^2}\,
  \frac{f(A_\ell)}{C_{\ell,\pm}},
\qquad
V_\ell(x)+xV_\ell'(x)=E_\ell(x),
\]
the lower-limit contributions cancel in \(V_\ell'\). Moreover,
\[
E_\ell'(x)
=\frac{8\ell^2f(A_\ell)}{x(4\ell^2-x^2)},
\qquad
J'(x)=-\frac{g'(x)}x.
\]
Direct differentiation, justified by the uniform domination (8.14), now
yields

\[
H(x)
=g(x)
+x\left[
J(x)-B_0
+\sum_{\ell=1}^{\infty}(-1)^{\ell-1}V_\ell(x)
\right],
\tag{8.20}
\]

\[
H'(x)
=J(x)-B_0
+\sum_{\ell=1}^{\infty}(-1)^{\ell-1}E_\ell(x),
\tag{8.21}
\]

\[
H''(x)
=\frac1x
\left[
\sum_{\ell=1}^{\infty}
(-1)^{\ell-1}
\frac{8\ell^2f(2\ell b)}
     {4\ell^2-x^2}
-g'(x)
\right].
\tag{8.22}
\]

The derivative of \(g\) also has the uniformly convergent image expansion

\[
g'(r)
=\sum_{j=0}^{\infty}
(-1)^j
\left[
f(2j+1-r)-f(2j+1+r)
\right].
\tag{8.23}
\]

It follows by expanding \(f\) into exponentials and summing the geometric
image series; absolute uniform convergence holds on any compact
subinterval of \([0,1)\), in particular on the integration intervals used by
the certificate away from the endpoint \(r=1\). At \(r=1\), the resulting
image series remains absolutely convergent because its \(j=0\) term is
finite and its remaining terms decay geometrically.

\hypertarget{explicit-left-image-tails}{%
\subsection{Explicit left-image tails}\label{explicit-left-image-tails}}

Since \(\pi>3\),

\[
f(u)\leq2e^{-3u/2},
\qquad
\chi(u)\leq\frac43e^{-3u/2}.
\tag{8.24}
\]

The integrand form of the difference in (8.13) gives

\[
\begin{aligned}
0\leq V_\ell
&=
\int_0^\infty
\frac{
\chi(A_\ell+C_{\ell,-}t)
-\chi(A_\ell+C_{\ell,+}t)
}{(1+t)^2}\,dt\\
&\leq
2\int_0^\infty
t\,f(A_\ell+C_{\ell,-}t)\,dt\\
&\leq
4e^{-3A_\ell/2}
\int_0^\infty
t\,e^{-3C_{\ell,-}t/2}\,dt\\
&=
\frac{16e^{-3\ell b}}{9C_{\ell,-}^2}
\leq
\frac{8e^{-3\ell b}}{9(4\ell-1)}.
\end{aligned}
\tag{8.25}
\]

Here the factor \(2t\) is the separation of the two image arguments, and
\(C_{\ell,-}=2\ell/x-1\geq4\ell-1\geq2\). This derives the algebra behind
\(V_\ell\) rather than merely quoting its bound. Because

\[
e^{-3b}\leq e^{-3}<\frac18,
\]

summing the geometric majorant yields

\[
\sum_{\ell=L}^{\infty}V_\ell
\leq
\frac{64e^{-3Lb}}{63(4L-1)}.
\tag{8.26}
\]

Similarly, since \(A_\ell+2\ell\pm1\geq4\ell\pm1\),

\[
J_{\ell,\pm}
\leq
\frac1{4\ell\pm1}
\int_{A_\ell}^{\infty}2e^{-3u/2}\,du
=\frac{4e^{-3\ell b}}
       {3(4\ell\pm1)}.
\tag{8.27}
\]

Therefore

\[
\sum_{\ell=L}^{\infty}E_\ell
\leq
\frac{64e^{-3Lb}}
     {21(4L-1)}.
\tag{8.28}
\]

Both left-image tails are explicit, uniform on the complete localization
interval, and compatible with the directed-rational certificate.

\hypertarget{strict-concavity-and-numerical-certification}{%
\subsection{Strict concavity and numerical certification}\label{strict-concavity-and-numerical-certification}}

For \(x\leq1/2\), the positive derivative spectral terms

\[
a_\beta(x)=\frac{2\sinh(\beta x)}{\cosh\beta}
\tag{8.29}
\]

decrease as \(\beta\) runs through \(\beta_j\). Indeed

\[
\frac{\partial}{\partial\beta}\log a_\beta(x)
=x\coth(\beta x)-\tanh\beta.
\tag{8.30}
\]

The first term increases with \(x\), so it is enough to take \(x=1/2\).
Writing \(t=\tanh(\beta/2)\), the desired comparison is equivalent to
\(t^2\geq1/3\). For \(\beta\geq\pi/2\), the elementary estimates
\(\pi>3\) and \(e^{3/2}>4>2+\sqrt3\) prove this inequality.
Consequently the alternating-series estimate gives

\[
g'(x)\geq
D(x)
=\frac{2\sinh(\pi x/2)}{\cosh(\pi/2)}
-\frac{2\sinh(3\pi x/2)}{\cosh(3\pi/2)}.
\tag{8.31}
\]

The certificate verifies the sufficient exact comparison

\[
f(1)>\frac{3f(3)}{f(3/2)},
\tag{8.32}
\]

which implies \(D'(x)>0\) throughout the localization interval. The positive
image terms in (8.22) decrease with \(\ell\), and their alternating sum is
bounded above by its first term

\[
C(x)=\frac{8f(2(1-x)/x)}{4-x^2},
\tag{8.33}
\]

which increases with \(x\). Partition \([1/4,1/2]\) into the four cells

\[
\left[\frac4{16},\frac5{16}\right],
\quad
\left[\frac5{16},\frac6{16}\right],
\quad
\left[\frac6{16},\frac7{16}\right],
\quad
\left[\frac7{16},\frac8{16}\right].
\]

Directed rational interval arithmetic proves on each cell \([a,b]\) that

\[
D(a)-C(b)-\frac b2>0,
\tag{8.34}
\]

with the smallest certified margin enclosed by

\[
0.018974256407
<
\min_{\text{four cells}}
\left(D(a)-C(b)-\frac b2\right)
<
0.018974256408.
\tag{8.35}
\]

Equations (8.22) and (8.34) imply

\[
\boxed{
H''(x)<-\frac12
\quad\text{for every }x\in[1/4,1/2].
}
\tag{8.36}
\]

To control every quadrature error, write \(c=\pi/2<2\) and
\(t=\tanh(cu)\). Direct differentiation gives

\[
\begin{aligned}
f'(u)&=-cf(u)t,\\
f''(u)&=c^2f(u)(2t^2-1),\\
f'''(u)&=c^3f(u)(5t-6t^3),\\
f''''(u)&=c^4f(u)(5-28t^2+24t^4).
\end{aligned}
\tag{8.37}
\]

Thus one may use

\[
|f'|\leq2f,
\qquad
|f''|\leq4f,
\qquad
|f'''|\leq16f,
\qquad
|f''''|\leq80f.
\tag{8.38}
\]

For \(c_0>0\), Leibniz's rule therefore gives

\[
\left|
\left(\frac{f(u)}{u+c_0}\right)^{(4)}
\right|
\leq
f(u)
\left[
\frac{80}{u+c_0}
+\frac{64}{(u+c_0)^2}
+\frac{48}{(u+c_0)^3}
+\frac{48}{(u+c_0)^4}
+\frac{24}{(u+c_0)^5}
\right].
\tag{8.39}
\]

On an interval of width \(w\) split into an even number \(n\) of equal
subintervals, the composite Simpson remainder is bounded by

\[
\frac{w^5}{180n^4}
\sup|f^{(4)}_{\mathrm{integrand}}|.
\tag{8.40}
\]

The truncated image expansion (8.23) has tail bounded by

\[
\frac{16}{7}\,8^{-M},
\tag{8.41}
\]

and integration against \(1/r\) from \(x\) to \(1\) adds at most

\[
\frac{1-x}{x}\frac{16}{7}\,8^{-M}.
\tag{8.42}
\]

The certificate encloses \(\pi\) by Machin's formula

\[
\pi=16\arctan(1/5)-4\arctan(1/239),
\tag{8.43}
\]

using alternating rational arctangent series, and encloses exponentials by
range reduction and alternating rational series. All arithmetic is integer,
rational, or directed fixed-point integer arithmetic; no floating-point
estimate is used as a proof.

It produces

\[
H'(0.363424228)
\in
[0.000000000766586,\ 0.000000000958142],
\tag{8.44}
\]

\[
H'(0.363424229)
\in
[-0.000000000275818,\ -0.000000000084262].
\tag{8.45}
\]

Global localization (8.9), strict concavity (8.36), and these opposite
derivative signs prove that the unique global maximizer \(x_*\) satisfies

\[
\boxed{
0.363424228<x_*<0.363424229.
}
\tag{8.46}
\]

The corresponding certified maximum is

\[
\boxed{
0.57155664213<H(x_*)<0.57155664221.
}
\tag{8.47}
\]

For comparison, deterministic continuum Voronoi allocation against an
endpoint has payoff

\[
F_{\mathrm{Vor}}(x,0)=1-\frac x2
\qquad(0<x\leq1),
\qquad
F_{\mathrm{Vor}}(0,0)=\frac12.
\tag{8.47a}
\]

Its supremum is one as \(x\downarrow0\), whereas the competing-Brownian game
has the unique interior optimum (8.46) and the strictly smaller value
(8.47). The two strategic allocation rules are therefore genuinely
different even on the interval.

Finally, if \(a_N\) is any endpoint best response on \(P_{N+1}\), compactness,
uniform convergence (7.3), and uniqueness of the continuum maximizer imply

\[
\frac{a_N}{N}\longrightarrow x_*.
\tag{8.48}
\]

Neither uniqueness of the finite-path maximizer nor a convergence rate is
claimed.

\appendix
\hypertarget{appendix-a.-complete-multi-hub-product-chain-and-equilibrium-classification}{%
\section{Complete multi-hub product-chain and equilibrium classification}%
\label{app:multihub}%
\label{appendix-a.-complete-multi-hub-product-chain-and-equilibrium-classification}}

\hypertarget{a.1.-automorphism-reduction-and-rate-one-transition-weights}{%
\subsection{Automorphism reduction and rate-one transition weights}\label{a.1.-automorphism-reduction-and-rate-one-transition-weights}}

Graph automorphisms interchange any two hubs and any two leaves, including
leaves belonging to different hubs. Therefore

\[
A_G(u_i,u_j)=0,
\qquad
A_G(\ell,\ell')=0,
\tag{A.1}
\]

and the remaining centered payoff types are

\[
h_{k,r}=A_G(u_i,\ell)
\quad(\ell\in L_i),
\qquad
A_G(u_i,\ell)=-d_{k,r}
\quad(\ell\in L_j,\ j\neq i).
\tag{A.2}
\]

For a hub target, the nonabsorbing single-walk orbit classes are

\[
A=\text{leaves attached to the target hub},
\qquad
B=\text{other hubs},
\qquad
C=\text{their leaves}.
\]

For a leaf target, they are

\[
H=\text{the target's hub},
\quad
K=\text{other hubs},
\quad
A=\text{other leaves at }H,
\quad
B=\text{leaves at }K.
\]

Multiply all single-walk transition probabilities by

\[
\Delta=k+r-1.
\tag{A.3}
\]

A hub-to-neighbor edge then has weight one, whereas a leaf-to-hub transition
has weight \(\Delta\), since the leaf jumps to its unique neighbor with
probability one. Thus the transition weights retain the constant-speed
boundary convention.

For the hub target, the weighted transition lists are

\[
A\longrightarrow\text{target}:\Delta,
\]

\[
B\longrightarrow\text{target}:1,
\qquad
B\longrightarrow B:k-2,
\qquad
B\longrightarrow C:r,
\]

\[
C\longrightarrow B:\Delta.
\tag{A.4}
\]

For the leaf target they are

\[
H\longrightarrow\text{target}:1,
\qquad
H\longrightarrow A:r-1,
\qquad
H\longrightarrow K:k-1,
\]

\[
K\longrightarrow H:1,
\qquad
K\longrightarrow K:k-2,
\qquad
K\longrightarrow B:r,
\]

\[
A\longrightarrow H:\Delta,
\qquad
B\longrightarrow K:\Delta.
\tag{A.5}
\]

At \(r=1\), the leaf-target class \(A\) is formally empty. Retaining it makes
the symbolic polynomial matrices uniform: all transitions into \(A\) have
weight \(r-1=0\), and every physical starting state belongs to the closed
subsystem on \(H,K,B\). Thus the formal extra rows cannot affect a physical
hitting probability or payoff.

Antisymmetry reduces the ordered pair problem to unordered distinct orbit
pairs, while equal-orbit pairs have value \(1/2\). In pair order

\[
(A,B),(A,C),(B,C),
\]

the hub-target system is

\[
M_{\mathrm h}
=\begin{pmatrix}
2(k+2r)&-2r&0\\
-2\Delta&4\Delta&0\\
0&0&2(k+2r)
\end{pmatrix},
\qquad
b_{\mathrm h}
=\begin{pmatrix}
2\Delta\\
2\Delta\\
k+2r+1
\end{pmatrix}.
\tag{A.6}
\]

In pair order

\[
(H,K),(H,A),(H,B),(K,A),(K,B),(A,B),
\]

the leaf-target system is

\[
M_{\ell}
=\begin{pmatrix}
2(k+2r)&0&-2r&2(r-1)&0&0\\
0&4\Delta&0&-2(k-1)&0&0\\
-2\Delta&0&4\Delta&0&-2(k-1)&-2(r-1)\\
2\Delta&-2&0&2(k+2r)&0&2r\\
0&0&-2&0&2(k+2r)&0\\
0&0&-2\Delta&2\Delta&0&4\Delta
\end{pmatrix},
\tag{A.7}
\]

\[
b_{\ell}
=\begin{pmatrix}
k+2r\\
k+2r\\
2\\
2(k+2r-1)\\
k+2r-1\\
2\Delta
\end{pmatrix}.
\tag{A.8}
\]

Let \(q_{\mathrm h}\) and \(q_\ell\) denote the corresponding solutions,
extended by

\[
q(S,S)=\frac12,
\qquad
q(T,S)=1-q(S,T).
\tag{A.9}
\]

Counting target orbits gives the exact centered-payoff identities

\[
\boxed{
\begin{aligned}
h_{k,r}
&=
1+(r-1)q_\ell(H,A)
+(k-1)q_{\mathrm h}(B,C)\\
&\quad
+(k-1)r\,q_\ell(K,B)
-\frac{k(r+1)}2.
\end{aligned}
}
\tag{A.10}
\]

\[
\boxed{
\begin{aligned}
d_{k,r}
&=
1+q_{\mathrm h}(A,B)
+(k-2)q_{\mathrm h}(C,B)\\
&\quad
+(r-1)q_\ell(A,K)
+r\,q_\ell(B,H)
+(k-2)r\,q_\ell(B,K)
-\frac{k(r+1)}2.
\end{aligned}
}
\tag{A.11}
\]

Both formulas include the final subtraction of half the graph order.

\hypertarget{a.2.-exact-bivariate-identities-and-coefficient-positivity}{%
\subsection{Exact bivariate identities and coefficient positivity}\label{a.2.-exact-bivariate-identities-and-coefficient-positivity}}

Set

\[
p=k-2,
\qquad
s=r-2.
\tag{A.12}
\]

For a coefficient-row array \(C=(C_0,C_1,\ldots)\), define

\[
\mathsf P_C(p,s)
=\sum_{i\geq0}\sum_{j\geq0}(C_i)_j\,p^is^j.
\tag{A.13}
\]

The following compact table gives all coefficients needed for the
parameter-uniform formulas; each parenthesized row is ordered by increasing
power of \(s\).

\begin{table}[htbp]
\centering
\caption{Positive coefficient rows for the all-parameter multi-hub formulas.}
\label{tab:multihub-coefficients}
\small
\renewcommand{\arraystretch}{1.08}
\begin{tabular}{@{}ccl@{}}
\toprule
Polynomial & Power of \(p\) & Coefficients in increasing powers of \(s\)\\
\midrule
\(\mathsf L\) & 0 & \((7248,9904,5052,1140,96)\)\\
              & 1 & \((9368,10448,4122,656,32)\)\\
              & 2 & \((4829,4109,1108,92)\)\\
              & 3 & \((1239,713,98)\)\\
              & 4 & \((158,46)\)\\
              & 5 & \((8)\)\\
\addlinespace
\(\mathsf H\) & 0 & \((18480,32872,23136,8068,1396,96)\)\\
              & 1 & \((32552,48080,27104,7154,848,32)\)\\
              & 2 & \((23003,27187,11534,2048,124)\)\\
              & 3 & \((8441,7493,2126,190)\)\\
              & 4 & \((1709,1013,144)\)\\
              & 5 & \((182,54)\)\\
              & 6 & \((8)\)\\
\addlinespace
\(\mathsf D\) & 0 & \((21888,73632,89888,54104,17392,2872,192)\)\\
              & 1 & \((40224,122224,129728,65640,17004,2124,96)\)\\
              & 2 & \((30888,81676,71654,28158,5088,340)\)\\
              & 3 & \((12534,27907,18964,5117,478)\)\\
              & 4 & \((2818,5094,2402,334)\)\\
              & 5 & \((332,464,116)\)\\
              & 6 & \((16,16)\)\\
\bottomrule
\end{tabular}
\end{table}

Exact determinant evaluation in \(\mathbb Z[p,s]\) gives

\[
\det M_{\mathrm h}
=8\Delta(k+2r)(2k+3r),
\tag{A.14}
\]

\[
\det M_{\ell}
=64\Delta\,\mathsf L(p,s).
\tag{A.15}
\]

Cramer's rule in (A.10)--(A.11) gives

\[
h_{k,r}
=\frac{\mathsf H(p,s)}
       {2(k+2r)\mathsf L(p,s)},
\tag{A.16}
\]

\[
d_{k,r}
=\frac{\mathsf D(p,s)}
       {2(k+2r)(2k+3r)\mathsf L(p,s)}.
\tag{A.17}
\]

The polynomials \(\mathsf L,\mathsf H,\mathsf D\) contain, respectively,

\[
20,\qquad27,\qquad34
\tag{A.18}
\]

strictly positive coefficients. Therefore, for every \(k\geq2,r\geq2\),

\[
\mathsf L(p,s)>0,
\qquad
h_{k,r}>0,
\qquad
d_{k,r}>0.
\tag{A.19}
\]

This is a proof on the entire unbounded parameter quadrant, not an inference
from sampled graph sizes.

For \(r=1\), substitution followed by the shift \(p=k-2\) gives

\[
\begin{aligned}
\mathsf L(p,-1)
&=1352+2418p+1736p^2+624p^3+112p^4+8p^5,\\
\mathsf H(p,-1)
&=1976+5238p+5426p^2+2884p^3\\
&\quad+840p^4+128p^5+8p^6,\\
\mathsf D(p,-1)
&=-1248-2936p-2544p^2-1048p^3\\
&\quad-208p^4-16p^5.
\end{aligned}
\tag{A.20}
\]

Consequently

\[
h_{k,1}>0,
\qquad
d_{k,1}<0.
\tag{A.21}
\]

\hypertarget{a.3.-complete-optimal-strategy-classification}{%
\subsection{Complete optimal-strategy classification}\label{a.3.-complete-optimal-strategy-classification}}

For a probability measure \(\mu\), let

\[
x_i=\mu(u_i),
\qquad
y_i=\sum_{\ell\in L_i}\mu(\ell),
\qquad
X=\sum_{i=1}^kx_i,
\qquad
Y=\sum_{i=1}^ky_i,
\qquad
X+Y=1.
\tag{A.22}
\]

When \(r\geq2\), abbreviate \(h=h_{k,r}>0\), \(d=d_{k,r}>0\), and

\[
t_{k,r}=\frac{d}{h+d}.
\tag{A.23}
\]

Testing \(\mu\) against a leaf in \(L_i\) and then against hub \(u_i\)
respectively gives the necessary and sufficient inequalities

\[
(h+d)x_i\geq dX,
\qquad
(h+d)y_i\leq dY,
\qquad
i=1,\ldots,k.
\tag{A.24}
\]

Equivalently,

\[
x_i\geq t_{k,r}X,
\qquad
y_i\leq t_{k,r}Y.
\tag{A.25}
\]

The distribution within each leaf class is arbitrary.

If

\[
kt_{k,r}<1,
\tag{A.26}
\]

summing the second family of (A.25) forces \(Y=0\). Conversely every hub
distribution satisfying the first family is optimal, so

\[
\mathcal O(G_{k,r})
=\left\{
\sum_{i=1}^kx_i\delta_{u_i}:
\sum_{i=1}^kx_i=1,\
x_i\geq t_{k,r}
\text{ for every }i
\right\}.
\tag{A.27}
\]

Every optimal strategy has support exactly \(k\).

If

\[
kt_{k,r}>1,
\tag{A.28}
\]

summing the first family forces \(X=0\), and

\[
\mathcal O(G_{k,r})
=\left\{
\mu\in\mathcal P\left(\bigcup_{i=1}^kL_i\right):
\mu(L_i)\leq t_{k,r}
\text{ for every }i
\right\}.
\tag{A.29}
\]

The exact minimum support is

\[
s_{\min}(k,r)
=\left\lceil\frac1{t_{k,r}}\right\rceil
=\left\lceil\frac{h_{k,r}+d_{k,r}}{d_{k,r}}\right\rceil.
\tag{A.30}
\]

Indeed at least this many distinct leaf classes are required by the upper
bound \(t_{k,r}\); assigning mass to one leaf in each of that many classes
attains it.

If

\[
kt_{k,r}=1,
\tag{A.31}
\]

both sums in (A.25) force equality term by term. The complete optimal set is

\[
\mathcal O(G_{k,r})
=\left\{
\mu:
\mu(u_i)=\frac Xk,\
\mu(L_i)=\frac{1-X}{k}
\text{ for all }i,\
0\leq X\leq1
\right\},
\tag{A.32}
\]

again with arbitrary splitting within each \(L_i\). A strategy with
\(0<X<1\) uses every hub and at least one leaf in every leaf class.

For \(r=1\), (A.21) says that every hub strictly defeats every leaf, while
hubs tie each other. Testing against the hubs excludes all positive leaf mass.
Thus

\[
\mathcal O(G_{k,1})
=\mathcal P(\{u_1,\ldots,u_k\}),
\tag{A.33}
\]

\[
\operatorname{PNE}(G_{k,1})
=\{u_1,\ldots,u_k\}^2.
\tag{A.34}
\]

For \(r\geq2\), every hub loses to leaves attached to other hubs, and every
leaf loses to its own hub. Therefore

\[
\operatorname{PNE}(G_{k,r})=\varnothing
\qquad(r\geq2).
\tag{A.35}
\]

In all parameter regimes,

\[
\operatorname{MNE}(G_{k,r})
=\mathcal O(G_{k,r})\times\mathcal O(G_{k,r}).
\tag{A.36}
\]

\hypertarget{a.4.-two-leaf-specialization-double-star-corollary-and-benchmarks}{%
\subsection{Two-leaf specialization, double-star corollary, and benchmarks}\label{a.4.-two-leaf-specialization-double-star-corollary-and-benchmarks}}

At \(r=2\), exact polynomial specialization gives

\[
\mathsf L(k-2,0)=(k+2)^2W_k,
\tag{A.37}
\]

\[
\mathsf H(k-2,0)=(k+2)^2H_k,
\tag{A.38}
\]

\[
\mathsf D(k-2,0)=2(k+2)^2D_k.
\tag{A.39}
\]

Substitution into (A.16)--(A.17) proves (3.6), while (3.7)--(3.8) imply

\[
kt_{k,2}<1
\quad\text{for every }k\geq2.
\tag{A.40}
\]

Equation (A.27) then proves the linear-support theorem.

The graph \(G_{2,r}\) is a symmetric double star. Let

\[
B_r=24r^3+69r^2+60r+16.
\tag{A.41}
\]

Substitution of \(k=2\) into the all-parameter identities gives

\[
h_{2,r}
=\frac{r(24r^3+85r^2+100r+38)}
       {4(r+1)B_r},
\tag{A.42}
\]

\[
d_{2,r}
=\frac{r(24r^3+23r^2-46r-40)}
       {2(3r+4)B_r}.
\tag{A.43}
\]

For \(r\geq2\), direct subtraction gives

\[
h_{2,r}-d_{2,r}
=\frac{
r(24r^4+257r^3+686r^2+686r+232)
}{
4(r+1)(3r+4)B_r
}
>0,
\tag{A.43a}
\]

so \(2t_{2,r}<1\), and the optimal two-hub interval is

\[
\mathcal O(G_{2,r})
=\left\{
p\delta_{u_1}+(1-p)\delta_{u_2}:
t_{2,r}\leq p\leq1-t_{2,r}
\right\},
\tag{A.44}
\]

and there is no pure equilibrium. At \(r=1\),

\[
G_{2,1}=P_4,
\qquad
h_{2,1}=\frac{19}{104},
\qquad
d_{2,1}=-\frac3{182},
\tag{A.45}
\]

in agreement with (6.8); both central hubs are safe and every hub-supported
mixture is optimal.

Table~\ref{tab:multihub-benchmarks} lists independent exact benchmark instances.

\begin{table}[htbp]
\centering
\caption{Independent exact multi-hub payoff and support benchmarks.}
\label{tab:multihub-benchmarks}
\resizebox{\textwidth}{!}{%
\begin{tabular}{@{}rrr r l r@{}}
\toprule
\(k\) & \(r\) & \(h_{k,r}\) & \(d_{k,r}\)
& Regime & Minimum exact support\\
\midrule
2 & 1 & \(19/104\) & \(-3/182\) & hub simplex & 1\\
2 & 2 & \(385/1812\) & \(19/755\) & all hubs & 2\\
3 & 2 & \(3375/12796\) & \(1087/38388\) & all hubs & 3\\
3 & 10 & \(27949935/109474204\) & \(128790385/985267836\) & leaves & 3\\
4 & 2 & \(2563/8592\) & \(2659/90216\) & all hubs & 4\\
4 & 5 & \(102911/369236\) & \(593869/6369321\) & leaves & 4\\
4 & 16 & \(3190937/12214968\) & \(4981210/32064291\) & leaves & 3\\
5 & 2 & \(559/1728\) & \(17/576\) & all hubs & 5\\
5 & 6 & \(12942633/44657776\) & \(32661159/312604432\) & leaves & 4\\
6 & 2 & \(1071/3124\) & \(1018/35145\) & all hubs & 6\\
\bottomrule
\end{tabular}%
}
\end{table}

These examples illustrate the distinct exact regimes. The parameter-uniform
claims themselves follow from the symbolic determinant identities, positive
shifted coefficients, and optimization argument above.

\hypertarget{appendix-b.-reproducibility-and-proof-versus-regression-boundaries}{%
\section{Computational verification and proof boundaries}%
\label{app:reproducibility}%
\label{appendix-b.-reproducibility-and-proof-versus-regression-boundaries}}

The accompanying supplementary material contains four standalone Python
programs and reproduction instructions. Each program uses only the Python
standard library, requires no network access or computer-algebra system, and
can be executed from its containing directory:
\begin{CertificateOutput}
python certify_endpoint_optimizer.py
python continuum_endpoint_uniqueness_certificate.py
python finite_path_aggregate_certificate.py
python multi_hub_support_certificate.py
\end{CertificateOutput}

The evidentiary roles of their computations are distinct.

\noindent\textbf{Parameter-uniform analytic arguments.}
The reflecting-boundary maximum principles, folding couplings, killed-walk
reflection, stopped product martingale, aggregate rectangle inequalities,
first-passage invariance principle, continuity and compactness argument, and
image-series domination are mathematical proofs valid for all stated
parameters. Finite numerical sweeps are not used to establish any of these
infinite families.

\noindent\textbf{Parameter-uniform symbolic identities.}
The finite-path and multi-hub programs implement exact integer polynomial
arithmetic and verify identities directly in polynomial rings. These include
the shifted positive-coefficient expansions in (5.26) and (5.32)--(5.35), the
two multi-hub determinant and Cramer-rule formulas, all coefficient
inventories, and the symbolic specializations to one leaf, two leaves, and
two hubs. Equality of finite coefficient arrays is an exact proof of each
polynomial identity for every parameter value.

\noindent\textbf{Theorem-critical finite exceptions.}
The three exceptional parameter pairs are \((a,d)=(1,2),(1,3),(2,2)\).
Only \((2,2)\) requires the exceptional explicitly solved Dirichlet rectangle,
certified by exact substitution into all 25 interior equations and giving
\(J_{4,2,1}=31/495\). The twelve rational
bridge inequalities are checked separately for \(12\leq m\leq23\); their
representative minimum \(B_{12}>0\) appears in (6.7). For
\(2\leq m\leq11\), exact product-chain solutions and back-substitution verify
all twenty endpoint comparisons for the two even-path centers. These are
explicit finite components of the proof, rather than evidence extrapolated
to other parameters.

\noindent\textbf{Certified numerical enclosures.}
The global-localization program proves (8.6)--(8.9) using rational spectral
enclosures and target-cell comparisons. The endpoint-uniqueness program uses
Machin's formula, directed integer arithmetic, the explicit image tails, and
composite Simpson remainder bounds to certify (8.34)--(8.47). Its output
therefore certifies rigorous intervals, not floating-point estimates.

\noindent\textbf{Finite regression sweeps.}
Additional killed-walk, rectangle, reflecting-boundary, lumped-chain, and
full-product-chain checks are reproducibility and implementation tests only.
In particular, the 630 lumped multi-hub comparisons and 344 full-vertex
comparisons do not prove parameter-uniform statements; those follow from the
symbolic identities and positivity arguments above. Full execution output and
integrity information belong to the accompanying reproducibility supplement,
not to the mathematical argument.

\bibliographystyle{plainnat}
\bibliography{strategic_competing_walks_references}

\end{document}